\documentclass[12pt,reqno]{article}
\usepackage{epsfig,amssymb,latexsym,amsmath,pifont,multicol,mathtools,amsmath,verbatim,amsthm,float,lipsum}
\usepackage{caption} 
\usepackage[utf8]{inputenc}
\usepackage[T1]{fontenc}

\usepackage[varg]{txfonts}
\let\mathbb=\varmathbb

\usepackage[sans]{dsfont}

\usepackage[left=2cm, right=2cm, top=2cm, bottom=2cm]{geometry}

\usepackage[%
cal=euler,
bb=fourier,
scr=zapfc,
]
{mathalfa}

\usepackage[font=small,labelfont=bf]{caption}
\usepackage{subfigure}
\usepackage{graphicx}
\graphicspath{{Figures/}} 
\usepackage{acronym}
\usepackage{latexsym}
\usepackage{paralist}
\usepackage{xspace}
\usepackage{mdframed}
\usepackage{palatino,pxfonts}
\usepackage{authblk}
\usepackage{enumitem}
\usepackage[numbers,sort&compress]{natbib}

\usepackage[dvipsnames,svgnames]{xcolor}
\colorlet{MyBlue}{DodgerBlue!75!Black}
\colorlet{MyGreen}{DarkGreen!95!Black}

\usepackage{hyperref}
\hypersetup{
colorlinks=true,
linktocpage=true,
pdfstartview=FitH,
breaklinks=true,
pdfpagemode=UseNone,
pageanchor=true,
pdfpagemode=UseOutlines,
plainpages=false,
bookmarksnumbered,
bookmarksopen=false,
bookmarksopenlevel=1,
hypertexnames=true,
pdfhighlight=/O,
urlcolor=MyBlue!60!black,linkcolor=MyBlue!70!black,citecolor=DarkGreen!70!black, 
pdftitle={},
pdfauthor={},
pdfsubject={},
pdfkeywords={},
pdfcreator={pdfLaTeX},
pdfproducer={LaTeX with hyperref}
}

\numberwithin{equation}{section}  
\usepackage[sort&compress,capitalize,nameinlink]{cleveref}
\crefname{example}{Ex.}{Exs.}

\crefrangeformat{equation}{\upshape(#3#1#4)\textendash(#5#2#6)}

\usepackage{mathtools}
\usepackage[para,online,flushleft]{threeparttable}
\usepackage{multirow}

\usepackage{physics}
\usepackage{float}
\newcommand{\eps}{\varepsilon}
\DeclareMathOperator*{\argmin}{argmin}

\DeclareMathOperator{\zer}{zer}

\DeclareMathOperator{\dist}{dist}

\DeclareMathOperator{\dom}{dom}

\DeclareMathOperator{\gr}{graph}
\DeclareMathOperator{\Int}{int}

\DeclareMathOperator{\Id}{Id}

\renewcommand{\iff}{\Leftrightarrow}

\renewcommand{\emptyset}{\varnothing}
\newcommand{\eqdef}{\triangleq}
\newcommand{\wlim}{\rightharpoonup}
\newcommand{\qbox}[1]{\quad\hbox{#1}\quad}

\newcommand{\scrC}{\mathcal{C}}

\newcommand{\scrG}{\mathcal{G}}
\newcommand{\scrH}{\mathcal{H}}

\newcommand{\scrK}{\mathcal{K}}

\newcommand{\scrS}{\mathcal{S}}

\newcommand{\scrX}{\mathcal{X}}

\newcommand{\setC}{\mathsf{C}}

\newcommand{\R}{\mathbb{R}}

\newcommand{\N}{\mathbb{N}}

\DeclareMathOperator{\NC}{\mathsf{N}}

\newcommand{\bC}{{\mathbf{C}}}

\newcommand{\opA}{\mathsf{A}}
\newcommand{\opB}{\mathsf{B}}

\newcommand{\opD}{\mathsf{D}}
\newcommand{\opM}{\mathsf{M}}

\usepackage[ruled,vlined]{algorithm2e}
\usepackage{algpseudocode}								

\theoremstyle{plain}
\newtheorem{theorem}{Theorem}

\newtheorem*{corollary*}{Corollary}
\newtheorem{lemma}[theorem]{Lemma}
\newtheorem{proposition}[theorem]{Proposition}

\theoremstyle{definition}

\newtheorem*{definition*}{Definition}
\newtheorem*{problem*}{Problem}
\newtheorem{assumption}{Assumption}
\newtheorem*{SH}{Standing Hypothesis}

\theoremstyle{remark}
\newtheorem{remark}{Remark}
\newtheorem*{remark*}{Remark}
\newtheorem*{notation*}{Notational remark}
\newtheorem{example}{Example}

\numberwithin{theorem}{section}
\numberwithin{remark}{section}
\numberwithin{example}{section}

\DeclarePairedDelimiter{\inner}{\langle}{\rangle}

\title{Inertial forward-backward algorithm with exterior penalization and Tikhonov regularization}
\date{\today}

\author[1,2]{\small Siqi Qu}
\author[1]{\small Mathias Staudigl}
\author[2]{\small Juan Peypouquet} 

\affil[1]{\footnotesize Mannheim University, Department of Mathematics, B6 26, 68159 Mannheim, Germany\\
(\href{mailto:mathias.staudigl@uni-mannheim.de}{qu.siqi@uni-mannheim.de,mathias.staudigl@uni-mannheim.de})}

\affil[2]{\footnotesize Rijksuniversiteit Groningen, Faculty of Science and Engineering, Systems, Control and Optimization — Bernoulli Institute, Groningen, The Netherlands\\
(\href{mailto:j.g.peypouquet@rug.nl}{j.g.peypouquet@rug.nl})}

\begin{document}

\maketitle

\begin{abstract}
In a real Hilbertian setting, we develop in this paper numerical splitting techniques guaranteeing strong convergence to the least norm solution of constrained variational inequalities. We develop a multiscale inertial forward-backward splitting algorithm for solving constrained monotone inclusion problems with multiscale penalization and vanishing Tikhonov regularization. The proposed framework accommodates smooth, nonsmooth, and mixed smooth--nonsmooth penalty operators, providing a unified treatment of a broad class of constrained monotone inclusion problems. In this general framework, we establish weak convergence of the generated iterates. By introducing a discrete Tikhonov central path, we further prove strong convergence to the minimum-norm solution of the problem under a mild constraint qualification condition on the problem data.
\end{abstract}

\section{Introduction}
\label{sec:Intro}
 Let $\scrH$ be a real Hilbert space with inner product $\inner{\cdot,\cdot}$ and associated norm $\norm{\cdot}$.
This paper is concerned with the monotone inclusion problem 
\begin{equation}\label{eq:MI}\tag{P}
0\in \Phi x:= \opA x+\opD x+\NC_{\scrC}(x),
\end{equation}
where $\opA$ and $\opD$ are Hilbert space valued operators, and $\scrC\subseteq\scrH$ is constrained domain. This is a three-operator formulation of a general class of variational problems, where a constrained equilibrium of the sum of two maximally monotone operators $\opA+\opD$ is requested over a domain $\scrC$, which admits a efficiently implementable representation in terms of penalty operators. This abstract equilibrium problem has been the starting point for a burgeoning literature in mathematical optimization, addressing the resolution of constrained equilibrium problems using operator splitting techniques. The article \cite{Attouch:2010aa} was the starting point, with a long list of follow-up contributions \cite{AttCzarPey11,attouch2011prox,AttCzar18,Banert:2015aa,Bot:2014aa,noun2013forward,peypouquet2026asymptotic}, among many others.  
\begin{mdframed}
\begin{SH}\label{ass:Standing}
The following properties of the data of the splitting problem \eqref{eq:MI} shall be in place throughout the paper:
\begin{enumerate}
    \item[(SH1)] $\opA:\scrH\to2^{\scrH}$ is maximally monotone operator.
    \item[(SH2)] $\opD:\scrH\to\scrH$ is monotone and $\frac{1}{\eta}$-Lipschitz.
    \item[(SH3)] $\scrC:=\zer(\opB_{1})\cap\zer(\opB_{2})\neq\emptyset$ is the set of common zeroes of a $\mu$-cocoercive operator $\opB_{1}:\scrH\to\scrH$, and a (potentially set-valued) monotone operator $\opB_{2}:2^{\scrH}\to\scrH$.
    \item[(SH4)] $\Phi=\opA+\opD+\NC_{\scrC}$ is maximally monotone and $\scrS=\zer(\Phi)\neq\emptyset$.
\end{enumerate}
\end{SH}
\end{mdframed}
(SH3) is a decomposition property of the constrained domain over which an equilibrium is sought for. It is naturally satisfied by many practical problems studied in the literature \cite{attouch1996dynamical,AttCzar18,AttCzarPey11,cabot2005,Peypouquet:2012aa}, and we describe a couple of concrete instances below. Standing hypothesis (SH4) is a constraint qualification type condition, and a standard hypothesis in the context of constrained variational inequalities \cite{Bot:2014aa,Bot:2016aa}. General conditions ensuring this hypothesis can be found in \cite{BauCom16}.
\subsection{Motivating examples}
\begin{example}[Simple Bilevel Optimization]
\label{ex:SBP}
Let $f\in\bC^{1,1}_{1/\mu}(\scrH)$ and $r\in\Gamma_{0}(\scrH)$. A simple bilevel optimization problem \cite{sabach2017first} is formulated as  
\begin{align*}
 & \min f(x) +r(x)\\ 
\text{s.t.: } & x\in\scrC:=\argmin\{g(y):y\in\scrH\},
\end{align*}
where $g$ is a convex and Fr\'{e}chet differentiable function. By Fermat's optimality condition 
$$
x\in\scrC\iff 0=\nabla g(x).
$$
Hence, $\scrC=\zer(\nabla g)$ defines the set of solutions to the lower-level optimization problem over which a minimizer of the function $f+r$ is searched for. Identifying the data as $\opA=\partial r,\opD=\nabla f$ and $\opB=\nabla g$ gives rise to an inclusion problem of the form \eqref{eq:MI}.  
\end{example}

\begin{example}[Constrained Variational Inequalities]
\label{eq:CVI}
Let $\scrC\subset\scrH$ be a set for which we know a convex function $\Psi\in\bC^{1,1}_{L_{\Psi}}(\scrH)$ having the properties that $\Psi\geq 0$ and $\scrC=\Psi^{-1}(0)$. Then, $\opB=\nabla\Psi$ is a maximally monotone operator whose zero set is $\scrC$. Given a mapping $\opD:\scrH\to\scrH$ and a proper convex and lower semi-continuous function $h:\scrH\to(-\infty,\infty]$, we search for a solution of the variational inequality of the second kind  
\begin{equation}
\text{Find $\bar{x}\in\scrC$ such that } \inner{\opD(\bar{x}),x-\bar{x}}+h(x)-h(\bar{x})\geq 0 \qquad\forall x\in\scrC.
\end{equation}
This is equivalent to the inclusion 
$$
0\in \opD(\bar{x})+\partial h(\bar{x})+\NC_{\scrC}(x),
$$
and arises frequently in optimal control problems \cite{de2011optimal,de2016strong}. 
\end{example}

\begin{example}[Generalized Nash equilibrium problem]
\label{ex:GNEP}
The generalized Nash equilibrium problem (GNEP) in infinite dimensions is an important mathematical problem of distributed optimal control \cite{hintermuller2015generalized,dreves2018generalized,kanzow2019multiplier}. Let $i=1,2,\ldots,N$ denote a set of agents (players), each characterized by a real Hilbert space $(\scrH_{i},\inner{\cdot,\cdot}_{i})$, and a functions 
$$
d_{i}:\bigotimes_{i=1}^{N}\scrH_{i}\to\R,\;x=(x_{1},\ldots,x_{N})=(x_{i},x_{-i})\mapsto d_{i}(x_{i},x_{-i})
$$
assumed to be continuously differentiable, with a Lipschitz continuous gradient. Let $\scrH=\bigotimes\scrH_{i}$ denote the product space. For all $i\in\{1,\ldots,N\}$, let $\scrX_{i}\subseteq\scrH_{i}$ be nonempty closed convex subspace. The decisive difference to the classical Nash equilibrium problem (NEP) is the presence of coupling constraints. Following the majority of applications, we introduce such coupling constraints via a bounded linear operator $L:\scrH\to\scrG$, and define the game's feasible set 
$$
\scrC:=\{x\in\scrH\vert x_{i}\in\scrX_{i}\quad\forall 1\leq i\leq N,\text{ and }Lx \in \scrK\}
$$
where $\scrK\subseteq\scrG$ is a closed convex set. In most applications $\scrK$ is a closed convex cone of functions. The GNEP is to find an $N$-tuple $x^{*}=(x_{i}^{*})_{1\leq i\leq N}$ satisfying 
$$
x_{i}^{*}\in\argmin\{d_{i}(x_{i},x_{-i}^{*})+r_{i}(x_{i})\vert (x_{i},x^{*}_{-i})\in\scrC\}\qquad \forall i=1,\ldots,N.
$$
We can model this problem as an inclusion of type \eqref{eq:MI}, by identifying the involved operators as 
$$
\opA=\bigotimes_{1\leq i\leq n}\partial r_{i},\text{ and }\opD=(\nabla_{x_{1}}d_{1}(x),\ldots,\nabla_{x_{N}}d_{N}(x)).
$$
Note that the constrained domain of the GNEP allows for a further decomposition. Let $\opB_{2}:=\partial\delta_{\scrX}:\scrH\to(-\infty,\infty]$ the convex indicator function of the set $\scrX\subseteq\scrH$, and $\opB_{1}=\nabla \Psi_{1}$, where $\Psi_{1}(x)=\frac{1}{2}\dist(Lx,\scrK)^{2}$, we obtain $\scrC=\zer(\opB_{2})\cap\zer(\opB_{1})$. Effective numerical schemes should exploit this domain decomposition character in their iterations. 
\end{example}

\subsection{Related work and contribution}
The general approach for resolving \eqref{eq:MI} we pursue in this paper is the penalty method. Replacing the normal cone with the defining operators $\opB_{1}$ and $\opB_{2}$, and scaling them with a non-decreasing penalty sequence $(\beta_{n})_{n}\subset(0,\infty)$ (technically $\beta_{n}\to\infty$), we device a numerical splitting method which asymptotically approaches the solution set $\scrS$. Such exterior penalization techniques can be related to multiscale effects in associated continuous-time dynamical systems, which have recently been studied in our previous work \cite{peypouquet2026asymptotic}, building on seminal insights developed in the above cited references. From a practical perspective, the perhaps most important instance of the constrained variational inequality \eqref{eq:MI} is the simple bilevel optimization problem, described in Example \ref{ex:SBP}. In that application, the penalty technique has been popularized by the influential work \cite{Solodov:2007aa}. Most of the established numerical algorithms use the forward-backward splitting technique. The aim of this work is to endow the forward-backward penalty scheme from \cite{Bot:2014aa} with inertial effects in the spirit Nesterov's method for convex optimization. Besides inertial terms, our numerical strategies features multi-scale effects embodied by penalty terms and an injected Tikhonov regularization strategy. The main contributions of this paper are two asymptotic convergence results on the generated trajectory under well-established conditions on the Fitzpatrick function associated with the penalty operators. Specifically, the two results are the following:
\begin{enumerate}
    \item Under the listed standing hypothesis, we prove weak ergodic convergence of the inertial forward-backward penalty method to the solution set $\scrS$. This result is derived in the three paradigmatic scenarios: 
    \begin{enumerate}
        \item $\opB_{1}=0$ and $\opB_{2}$ maximally monotone;
        \item $\opB_{2}=0$ and $\opB_{1}$ $\mu$-cocoercive; 
        \item $\opB_{1}=\nabla\Psi_{1}$ and $\opB_{2}=\partial\Psi_{2}$ for convex functions $\Psi_{i}:\scrH\to\R\cup\{-\infty,\infty\}$, with $\Psi_{1}$ being Lipschitz smooth and $\Psi_{2}$ closed convex proper.
    \end{enumerate}
    \item If the (unique) least-norm element of the solution set $\scrS$ lies in the interior of the domain of the set-valued operator $\opA$, a fine analysis establishes strong convergence of the iterates generated by our algorithm to the least norm solution. 
\end{enumerate}
The above described results generalize recent weak and strong convergence results for penalty-regulated dynamical systems, designed to solve the splitting problem \eqref{eq:MI} via the addition of inertial effects. The closest contribution we could identify in the literature is the recent paper \cite{Bot:2019aa}. This article is also concerned with the study of forward-backward type splitting methods for the general constrained variational inequality \eqref{eq:MI}. The key difference between that paper and ours is the design of the numerical scheme. While \cite{Bot:2019aa} inject an extrapolated sequence into the resolvent operator, our numerical splitting strategy is formulated along the classical inertial algorithms inspired by Nesterov's accelerated gradient method \cite{nesterov2018lectures}. Additionally, our analysis explicitly allows for domain decomposition of the feasible set $\scrC$. The general decomposition hypothesis (SH3) suggests that the set-valued operator $\opB_{2}$ is included in the backward resolvent step, thus making it an explicitly time-dependent object. Consequently, our Lyapunov analysis is fundamentally different, both in the weak convergence as well as in the strong convergence theorems. Finally, our strong convergence argument relies on a detailed (and in our opinion novel) analysis of the central path associated with a sequence of Tikhonov regularized operators, which leverages the local topological structure of maximally monotone operators with non-empty interiors of their domains. From a mathematical perspective, this can be seen as our most intersting contribution.\\ 

The rest of this paper is organized as follows. Section \ref{sec:IFBwithTikhonov} settles the general notation employed in this article, and develops the weak ergodic convergence results. Section \ref{strongconvergence} starts with establishing general properties of a central path trajectory, representing the sequence of unique solutions of a family of Thikhonov regularized auxiliary operators. These preparatory results are then used in the proof of Theorem \ref{th:strongconvergence_main}, where we demonstrate the strong convergence of the produced algorithmic sequence. Section \ref{sec:numerics} contains numerical experiments performed with our method on the image inpainting problem.

\section{Asymptotic Analysis and Weak Convergence}
\label{sec:IFBwithTikhonov}
\subsection{Preliminaries}
\label{sec:prelims}

For the reader's convenience we present first some notations which are used throughout the paper. The symbols $\wlim$ and $\to$ denote weak and strong convergence, respectively. For a function $f:\scrH\to\bar{\R}$ we denote by $\dom(f)=\{x\in\scrH\vert f(x)<\infty\}$ its effective domain and say that $f$ is proper if $\dom(f)\neq\emptyset$ and $f(x)\neq-\infty$ for all $x\in\scrH$. If $f$ is convex, we let $\partial f(x)=\{u\in\scrH\vert f(y)\geq f(x)+\inner{y-x,v}\quad\forall y\in\scrH\}$ the subdifferential of $f$ at $x\in\dom(f)$. $\Gamma_{0}(\scrH)$ collects the totality of closed convex proper functions.

Let $\scrC\subseteq\scrH$ be a nonempty set. The indicator function of $\setC$, $\delta_{\setC}:\scrH\to\bar{\R}$, is the function satisfying $\delta_{\scrC}(x)=0$ if $x\in\scrC$ and $+\infty$ otherwise. The subdifferential of the indicator function is the normal cone 
\[
\NC_{\scrC}(x)\eqdef\left\{\begin{array}{ll} \{u\in\scrH\vert \inner{y-x,u}\leq 0\} & \text{if }x\in\scrC,\\ 
\emptyset & \text{else}.
\end{array}\right.
\]
The normal cone of a set $C\subset\scrH$ is defined as $\NC_{C}(x)=\{p\in\scrH\vert \inner{y-x,p}\leq 0\quad\forall y\in C\}$ if $x\in C$, and $\NC_{C}(x)=\emptyset$, otherwise.

For a set-valued operator $\opM:\scrH\to 2^{\scrH}$ we denote by $\gr(\opM)=\{(x,u)\in\scrH\times\scrH\vert u\in \opM(x)\}$ its graph, $\dom(\opM)=\{x\in\scrH\vert \opM(x)\neq\emptyset\}$ its domain, and by $\opM^{-1}:\scrH\to 2^{\scrH}$ its inverse, defined by 
\[
(u,x)\in \gr(\opM^{-1})\iff (x,u)\in\gr(\opM).
\]
We let $\gr(\opM)=\{(x,x^{*})\in\scrH\times\scrH\vert x^{*}\in \opM(x)\}$ be the graph of $\opM$ and $\zer(\opM)=\{x\in\scrH\vert 0\in \opM(x)\}$ denote the set of zeros of $\opM$.  $(x,p)\in\gr(\NC_{\scrC})$ is equivalent to the fact that the support function to the set $C$ in direction $p$, $\sigma_{\scrC}(p)=\sup_{y\in\scrC}\inner{y,p}$ satisfies $\sigma_{\scrC}(p)=\inner{x,p}$. An operator $\opM$ is monotone if $\inner{x-y,u-v}\geq 0$ for all $(x,u),(y,v)\in\gr(\opM)$. A monotone operator $\opM$ is maximally monotone if there exists no proper monotone extension of the graph of $\opM$ on $\scrH\times\scrH$. The resolvent of $\opM$, $J_{\opM}:\scrH\to 2^{\scrH}$ is defined by $p\in J_{\opM}(x) \iff x\in p+\opM(p)$. 

The Fitzpatrick function associated to a monotone operator $\opM$ will be the key to formulate tractable conditions to prove the asymptotic convergence of our scheme. It is defined as 
\begin{equation}\label{eq:FP}
\varphi_{\opM}:\scrH\times\scrH\to\bar{\R},(x,p)\mapsto \varphi_{\opM}(x,p):=\sup_{(y,v)\in\gr(\opM)}\{\inner{x,v}+\inner{y,p}-\inner{y,v}\} 
\end{equation}

\begin{lemma} \label{L:transportation_formula}
Let $f$ be convex and $\ell$-smooth, and let $u$ be an unconstrained minimizer of $f$. For every $c>0$, we have
$$2\big(f(z)-\min(f)\big)+2\langle u-y,\nabla f(y)\rangle \le (c-\frac{1}{\ell})\|\nabla f(y)\|^2+\left[\frac{1}{c}+\ell\right]\|z-y\|^2.$$
\end{lemma}
\begin{proof}
On the one hand, we have (cf. \cite[][Theorem 2.1.5]{nesterov2018lectures})
$$2f(u)\ge 2f(y)+2\inner{\nabla f(y),u-y}+\frac{1}{\ell}\|\nabla f(y)\|^2.$$
On the other,
$$2f(z) \le 2f(y)+2\inner{\nabla f(y),z-y}+\ell\|z-y\|^2  \le 2f(y)+c\|\nabla f(y)\|^2+\frac{1}{c}\|z-y\|^2+\ell\|z-y\|^2$$
The proof is completed by combining the two inequalities and rearranging the terms.
\end{proof}

\begin{lemma}\label{L:real1}
Consider a real sequence $(\rho_n)$, along with nonnegative sequences $(\omega_n)$, $(e_n)$, $(d_n)$, $(a_n)$ and $(b_n)$. Suppose that, 
\begin{equation}
\rho_{n+1} = \omega_{n+1}-a_n\omega_n+b_n
\end{equation}
Additionally, assume that there exists $n_{0}$ such that for every $n\ge n_0$, one has 
\begin{equation}\label{eq:RS}
\rho_{n+1}+e_n\le\rho_n+d_n. 
\end{equation}
Then the following statements hold true:
\begin{itemize}
    \item [(i)] 
    If $\sup_{n\ge n_0}a_n\le 1$ and $\sum_{n\ge n_0} d_n<\infty$, then $\lim_{n\to\infty}\rho_n$ exists, and $\sum_{n\ge n_0} e_n<\infty$. 
    \item [(ii)] If, moreover, $\lim_{n\to\infty}a_n=a<1$ and $\lim_{n\to\infty}b_n$ exists, then $\lim_{n\to\infty}\omega_n$ exists.
\end{itemize}
\end{lemma}
\begin{proof}
We consider first the long-time behavior of the sequence $(\rho_{n})_{n}$ satisfying the bound \eqref{eq:RS}. If $(d_{n})_{n}$ is summable and $\rho_n$ is bounded from below, the Robbins-Siegmund lemma implies that $\lim_{n\to\infty}\rho_{n}$ exists and $(e_{n})_{n}$ is summable. It only needs to be shown that indeed the sequence $(\rho_{n})_{n}$ is bounded from below. To see this, let us write $\Omega_n:=\rho_n-\sum_{j=0}^{n-1}d_j$, which is non-increasing and $\Omega_n\le\rho_n$. If $\rho_n$ is not bounded from below, for every $M>0$, there is $N$ such that $\Omega_n\le\Omega_N\le \rho_{N}\le-M$ for every $n\ge N$. It follows that  $\sum_{j=0}^{\infty}d_j-M\ge \Omega_{n+1}+\sum_{j=0}^{n}d_j\ge \rho_{n+1}\ge\omega_{n+1}-\omega_n$. Taking $M=1+\sum_{j=0}^{\infty}d_j$, we see that $\omega_{n+1}-\omega_n\le-1$ for every $n\ge N$, thus it cannot remain nonnegative, yielding a contradiction. This proves claim (i). 

We now prove the claimed convergence properties of $(\omega_{n})$. To that end, define 
$$
c_{n}:=\rho_{n+1}-b_{n}
$$
Since $\lim_{n}b_{n}$ is assumed to exist, we have $\lim_{n\to\infty}c_{n}=c=\rho-b$. We obtain $\omega_{n+1}=c_{n}+a_{n}\omega_{n}$. Let 
$$
L:=\frac{c}{1-a}\text{ and define }x_{n}:=\omega_{n}-L.
$$
In terms of these derived objects, we obtain the representation 
$$
x_{n+1}=a_{n}x_{n}+r_{n},\quad r_{n}:=(c_{n}-c)+(a_{n}-a)L.
$$
We have $r_{n}\to 0$ and $a_{n}\to a<1$. Hence, $x_{n}\to 0$, and therefore $\omega_{n}\to L.$
\end{proof}





\begin{lemma}[Opial] 
\label{lem:Opial}
Let $\scrS$ be a nonempty set of $\scrH$ and $(x_{n})_{n}$ be a sequence in $\scrH$ such that the following two conditions hold:
\begin{itemize}
\item[(i)] for every $x\in\scrS$, $\lim_{k\to\infty}\norm{x_{n}-x}$ exists;
\item[(ii)] every sequential weak cluster point of $(x_{n})$ is in $\scrS$;
\end{itemize}
Then $(x_{k})_{k}$ converges weakly to a point in $\scrS$.
\end{lemma}

\subsection{The splitting method}
Motivated by the second-order dynamical systems with penalty terms introduced in \cite{boct2018second} and by the inertial forward--backward framework of \cite{Bot:2019aa}, we propose a multiscale inertial forward--backward algorithm for solving \eqref{eq:MI} along the lines of \cite{peypouquet2026asymptotic}. To asymptotically enforce feasibility, we add to include into the backward step the set set-valued operator $\opB_{2}$, leading to a family of monotone operators $(\opA_{n}):\scrH\to 2^{\scrH}$, defined as $\opA_n:=\opA+\beta_n\opB_2$. We implicitly assume that the resolvent $(\Id+\opA_{n})^{-1}=J_{\opA_{n}}$ is a single-valued function, and easy to compute. General conditions for the sum of maximally monotone operators being maximally monotone can be found in \cite{BauCom16}. To induce stability and eventually establish our strong convergence results, we add to the forward step a scaled identity operator $\eps_{n}\Id_{\scrH}$, with $\eps_{n}$ being a positive and asymptotically vanishing operator. Additionally, we include the cocoercive penalty operator $\opB_{1}$ in to the forward step as well. In total, we obtain the monotone and Lipschitz continuous operator $V_{n}:\scrH\to\scrH$, defined by 
\[
V_n:=\opD+\beta_n\opB_1+\eps_{n}\Id_{\scrH}.
\]
With this splitting, we thus simultaneously incorporate inertial effects, penalization, and Tikhonov regularization. We consider the iterative scheme
\begin{equation} \label{E:Algorithm4}\tag{IFBT}
\left\{\begin{array}{rcl}
y_n & = & x_n+\alpha_n(x_n-x_{n-1}) \\
x_{n+1}& = & J_{\lambda_n\opA_n}\big(y_n-\lambda_nV_ny_n\big).
\end{array}\right.
\end{equation}
where $x_1=x_0\in\scrH$. For ease of notation, we set $\opD_{\eps_{n}}:=\opD+\eps_{n}\Id_{\scrH}$. This discrete algorithm preserves key features of the second-order continuous-time system, including adaptive damping and operator scaling, and provides a natural framework for analyzing convergence in the iterative setting.
\begin{remark}
Inertial effects in forward-backward splitting methods with penalty terms have been studied as well in \cite{Bot:2019aa}. We would like to emphasize that a key distinction from \cite{Bot:2019aa} lies in the evaluation of the forward operator. While the forward step in \cite{Bot:2019aa} is performed at the current iterate $(x_n)$, we evaluate it at the extrapolated point $(y_n)$,  yielding a genuinely extrapolated inertial forward--backward scheme.
\end{remark}
Throughout the paper, we make the following general assumptions on the long-run behavior of the parameter sequences; More specific assumptions will be introduced in the particular theoretical statements to follow. 
\begin{assumption}\label{ass:assumptionbasic} Let $(\lambda_n)_{n\geq1},(\beta_n)_{n\geq1}$ and $(\eps_n)_{n\geq1}$ be sequences of positive real numbers satisfying
\begin{itemize}
\item[(i)] $(\eps_n)_{n\geq1}$ is non-increasing and $\lim_{n\to \infty}\eps_n=0$;
\item[(ii)] $(\beta_n)_{n\geq1}$ is non-decreasing and $\lim_{n\to \infty}\eps_n=\infty$;
\item[(ii)] $(\lambda_n)_{n\geq1}\in \ell^2_{+}(\N)\setminus \ell^1_{+}(\N)$.
\end{itemize}
\end{assumption}

\subsection{Asymptotic properties of the iterates}

Now we establish a basic estimate for the iterates generated by Algorithm \eqref{E:Algorithm4}. This inequality will play a central role in the subsequent weak and strong convergence analysis.


\begin{lemma}\label{lem:discrete}
Let $(x_n)_{n\geq 0}$ be the sequence generated by Algorithm \eqref{E:Algorithm4}. Let $(u,w) \in \gr(\opA+\opD+\NC_{\scrC})$, so that $w=v+\opD u+p$ for some $v\in \opA u$ and $p\in \NC_{\scrC}(u)$. The following inequality holds for every $\gamma_1 >0$ and $n\geq 1$:
\begin{align*}
\norm{x_{n+1}-u}^2-\norm{u-x_n}^2 & \le \alpha_n\norm{u-x_n}^2-\alpha_n\norm{u-x_{n-1}}^2+(\alpha_n+\gamma_1-1)\norm{x_n-x_{n+1}}^2 \\
& \quad  +(2\alpha_n+\gamma_1)\norm{x_n-x_{n-1}}^2 +\frac{2\lambda_n^2(1+\alpha_n^2)}
{\gamma_1}\|\opD_{\eps_{n}}u+v\|^2 \\
& \quad + \frac{2\lambda_n^2\beta_n^2(1+\alpha_n^2)}
{\gamma_1}\|\opB_1 y_n\|^2  +2\lambda_n\left[\frac{\lambda_n(1+\alpha_n^2)(L+\eps_n)^2}
{\gamma_1}-\eps_n\right]\|y_n-u\|^2 \\
& \quad + 2\lambda_n\beta_n\inner{u-y_n,\opB_1 y_n} +2\lambda_n\beta_n\inner{u-x_{n+1},\xi_{n+1}}  \\
& \quad  + 2\lambda_n\inner{u-y_n,w-p+\eps_nu}.
\end{align*}
for some $\xi_{n+1}\in\opB_{2}(x_{n+1})$.
\end{lemma}

\begin{proof}
From \eqref{E:Algorithm4}, we have
$$y_n-\lambda_nV_ny_n-x_{n+1} \in \lambda_n \opA_n x_{n+1},$$ 
so there is $\xi_{n+1} \in \opB_2(x_{n+1})$ such that
$$y_n-\lambda_nV_ny_n-x_{n+1}-\lambda_n\beta_n\xi_{n+1} \in \lambda_n \opA x_{n+1}.$$ 
We also have $\lambda_{n} v \in\lambda_{n}\opA u$. The monotonicity of $\opA$ yields
$$\inner{x_{n+1}-u,y_n-x_{n+1}-\lambda_nV_ny_n-\lambda_n\beta_n\xi_{n+1}-\lambda_nv}\ge 0,
$$
which we rewrite as
\begin{equation}\label{eq:wi1}
2\inner{x_{n+1}-u,x_{n+1}-y_n}\leq 2\lambda_n\inner{u-x_{n+1},V_ny_n+v} + 2\lambda_n\beta_n\inner{u-x_{n+1},\xi_{n+1}}.
\end{equation}
Substituting $y_n  =  x_n+\alpha_n(x_n-x_{n-1})$, the left-hand side becomes
\begin{align*}
2\inner{x_{n+1}-u,x_{n+1}-y_n}&=2\inner{x_{n+1}-u,x_{n+1}-x_n-\alpha_n(x_n-x_{n-1})}\\
&=2\inner{x_{n+1}-u,x_{n+1}-x_n}  -2\alpha_n\inner{x_n-u,x_n-x_{n-1}} -2\alpha_n\inner{x_{n+1}-x_n,x_n-x_{n-1}} \\
& \ge  \norm{x_{n+1}-u}^2+\norm{x_{n+1}-x_n}^2-\norm{u-x_n}^2 - \alpha_n\norm{u-x_n}^2 - \alpha_n\norm{x_n-x_{n-1}}^2\\
&\quad +\alpha_n\norm{u-x_{n-1}}^2 -\alpha_n\norm{x_n-x_{n+1}}^2-\alpha_n\norm{x_n-x_{n-1}}^2,
\end{align*}
Combining this with \eqref{eq:wi1}, we obtain
\begin{align} \label{eq:wi4}
    \norm{x_{n+1}-u}^2-\norm{u-x_n}^2 \le &\ \alpha_n\norm{u-x_n}^2-\alpha_n\norm{u-x_{n-1}}^2+(\alpha_n-1)\norm{x_n-x_{n+1}}^2 + 2\alpha_n\norm{x_n-x_{n-1}}^2 \nonumber \\
    &  +2\lambda_n\inner{u-x_{n+1},V_ny_n+v}  +2\lambda_n\beta_n\inner{u-x_{n+1},\xi_{n+1}}.
\end{align}
For the second-to-last term on the right-hand side, we write 
$$u-x_{n+1}=u-y_n+y_n-x_n+x_n-x_{n+1}=u-y_n+\alpha_n(x_n-x_{n-1})+x_n-x_{n+1},$$
and then use the strong monotonicity of $\opD_{\eps_n}$, and Young's inequality to deduce that
\begin{equation}\label{eq:wi5}
\begin{split}
2\lambda_n\inner{u-x_{n+1},V_ny_n+v} 
& = 2\lambda_n\inner{u-y_n+\alpha_n(x_n-x_{n-1})+x_n-x_{n+1},\opD_{\eps_{n}}y_n+\beta_n\opB_1 y_n+v} \\
& = 2\lambda_n\inner{u-y_n,\opD_{\eps_{n}}y_n-\opD_{\eps_{n}}u+\opD_{\eps_{n}}u+\beta_n\opB_1 y_n+v} \\
&\quad + 2\lambda_n\inner{\alpha_n(x_n-x_{n-1})+x_n-x_{n+1},\opD_{\eps_{n}}y_n+\beta_n\opB_1 y_n+v} \\
& \le -2\lambda_n\eps_n\|y_n-u\|^2 + 2\lambda_n\inner{u-y_n,v+\opD u+\eps_nu} \\
& \quad + 2\lambda_n\beta_n\inner{u-y_n,\opB_1 y_n} + \gamma_1\|x_n-x_{n-1}\|^2 + \gamma_1\|x_{n+1}-x_n\|^2 \\
& \quad + \frac{\lambda_n^2(1+\alpha_n^2)}{\gamma_1}\|\opD_{\eps_{n}}y_n+\beta_n\opB_1 y_n+v\|^2, 
\end{split}
\end{equation}
for every $\gamma_1>0$. But
\begin{equation*}
\begin{split}
\|\opD_{\eps_{n}}y_n+\beta_n\opB_1 y_n+v\|^2 & = \|\opD_{\eps_{n}}y_n-\opD_{\eps_{n}}u+\opD_{\eps_{n}}u+\beta_n\opB_1 y_n+v\|^2 \\
& \le 2 \|\opD_{\eps_{n}}y_n-\opD_{\eps_{n}}u\|^2+2 \|\opD_{\eps_{n}}u+v\|^2+2\beta_n^2\|\opB_1 y_n\|^2 \\ 
& \le 2(L+\eps_n)^2\|y_n-u\|^2+2 \|\opD_{\eps_{n}}u+v\|^2+2\beta_n^2\|\opB_1 y_n\|^2.
\end{split}
\end{equation*} 
To conclude, we substitute this in \eqref{eq:wi5} and then in \eqref{eq:wi4}. 
\end{proof}

Lemma \ref{lem:discrete} gives us insight into the conditions on the parameters that will later ensure the convergence of Algorithm \eqref{E:Algorithm4}. First, let $(\alpha_n)_{n}$ be non-increasing and bounded from above by some $\alpha>0$. It follows that

\begin{align*}
\norm{x_{n+1}-u}^2-\norm{u-x_n}^2 & \le \alpha_n\norm{u-x_n}^2-\alpha_{n-1}\norm{u-x_{n-1}}^2+(\alpha+\gamma_1-1)\norm{x_n-x_{n+1}}^2 \\
& \quad  +(2\alpha+\gamma_1)\norm{x_n-x_{n-1}}^2 +\frac{2\lambda_n^2(1+\alpha_n)}
{\gamma_1}\|\opD_{\eps_{n}}u+v\|^2 \\
& \quad + \frac{2\lambda_n^2\beta_n^2(1+\alpha^2)}
{\gamma_1}\|\opB_1 y_n\|^2  +2\lambda_n\left[\frac{\lambda_n(1+\alpha^2)(L+\eps_n)^2}
{\gamma_1}-\eps_n\right]\|y_n-u\|^2 \\
& \quad + 2\lambda_n\beta_n\inner{u-y_n,\opB_1 y_n} +2\lambda_n\beta_n\inner{u-x_{n+1},\xi_{n+1}}  \\
& \quad  + 2\lambda_n\inner{u-y_n,w-p+\eps_nu}.
\end{align*}

Next, we would like the {\it velocity} terms to telescope. To this end, we choose $\gamma_1>0$ such that 
$$\alpha+\gamma_1-1\le -(2\alpha+\gamma_1)-\nu_1,$$
for some $\nu_1>0$. This is equivalent to $\gamma_1\le \frac{1-3\alpha-\nu_1}{2},$ which is possible if $\alpha<\frac{1}{3}$. We choose $\gamma_1$ with equality to make it as large as possible, and $\nu_1$ close to zero, but anyway less than $1-3\alpha$. 
Then, we set
$$c_\alpha:=\frac{\gamma_1}{1+\alpha^2}\qbox{and}d_\alpha:=2\alpha+\gamma_1=\frac{1+\alpha-\nu_1}{2}.$$
Introducing the notation
\begin{equation} \label{eq:rho_n}
\rho_n(u,d):=\norm{x_n-u}^2-\alpha_{n-1}\norm{u-x_{n-1}}^2+d\norm{x_n-x_{n-1}}^2,
\end{equation}
with $d\in\R$, we obtain
\begin{equation}\label{eq:discrete2}
\begin{split}
\rho_{n+1}(u,d_\alpha) &\leq \rho_n(u,d_\alpha) - \nu_1\norm{x_n-x_{n+1}}^2 + \frac{2\lambda_n^2}
{c_\alpha}\|\opD_{\eps_{n}}u+v\|^2 \\
& \quad +2\lambda_n\left[\frac{\lambda_n(L+\eps_n)^2}
{c_\alpha}-\eps_n\right]\|y_n-u\|^2 + \frac{2\lambda_n^2\beta_n^2}
{c_\alpha}\|\opB_1 y_n\|^2  \\
& \quad + 2\lambda_n\beta_n\inner{u-y_n,\opB_1 y_n} +2\lambda_n\beta_n\inner{u-x_{n+1},\xi_{n+1}}   + 2\lambda_n\inner{u-y_n,w-p+\eps_nu}.
\end{split}
\end{equation}
If, moreover,
$$\lambda_n\le\left[\frac{c_\alpha}{2(L+\eps_n)^2}\right]\eps_n,\qbox{where}c_\alpha=\frac{\gamma_1}{1+\alpha^2}<\frac{1-3\alpha}{2(1+\alpha^2)},$$ 
then
\begin{equation}\label{eq:discrete3}
\begin{split}
\rho_{n+1}(u,d_\alpha) &\leq \rho_n(u,d_\alpha) - \nu_1\norm{x_n-x_{n+1}}^2 - \lambda_n\eps_n\|y_n-u\|^2
+ \frac{2\lambda_n^2}
{c_\alpha}\|\opD_{\eps_{n}}u+v\|^2  + \frac{2\lambda_n^2\beta_n^2}
{c_\alpha}\|\opB_1 y_n\|^2  \\
& \quad + 2\lambda_n\beta_n\inner{u-y_n,\opB_1 y_n} +2\lambda_n\beta_n\inner{u-x_{n+1},\xi_{n+1}}   + 2\lambda_n\inner{u-y_n,w-p+\eps_nu}.
\end{split}
\end{equation}

Concerning the way the set $\scrC$ is approximated, we shall provide convergence results in three particular cases: (1) $\opB_1=0$ and $\opB_2$ is maximally monotone, (2) $\opB_1$ is cocoercive and $\opB_2=0$, and (3) $\opB_1=\nabla \Psi_1$ and $\opB_2=\partial\Psi_2$, where $\Psi_1$ is Lipschitz smooth and convex, and $\Psi_2\in\Gamma_{0}(\scrH)$.

In each case, we give bounds for the expression
$$\delta_n:=2\lambda_n\beta_n\inner{u-y_n,\opB_1 y_n} +2\lambda_n\beta_n\inner{u-x_{n+1},\xi_{n+1}}  - 2\lambda_n\inner{u-y_n,p}$$
on the right-hand side of \eqref{eq:discrete3}.\\

\noindent{\bf Case 1.} If $\opB_1\equiv 0$, then
    \begin{align*}
        \delta_n & = 2\lambda_n\beta_n\inner{u-x_{n+1},\xi_{n+1}}  - 2\lambda_n\inner{u-y_n,p} \\
        & = 2\lambda_n\beta_n\inner{u-x_{n+1},\xi_{n+1}} - 2\lambda_n\inner{u-x_{n+1},p} - 2\lambda_n\inner{x_{n+1}-y_n,p}  \\
        & = 2\lambda_n\beta_n\bigg[\inner{u-x_{n+1},\xi_{n+1}}+\inner{x_{n+1},\frac{p}{\beta_n}}-\inner{u,\frac{p}{\beta_n}}\bigg] - 2\lambda_n\inner{x_{n+1}-y_n,p} \\
        & \le 2\lambda_n\beta_n\left[\varphi_{\opB_2}\bigg(u,\frac{p}{\beta_n}\bigg)-\sigma_\scrC\bigg(\frac{p}{\beta_n}\bigg)\right] + \frac{2\lambda_n^2\|p\|^2}{\nu_1} 
        + \frac{\nu_1}{2}\|x_{n+1}-x_n\|^2 +\frac{\alpha^2\nu_1}{2}\|x_n-x_{n-1}\|^2. 
    \end{align*}
    Inequality \eqref{eq:discrete3} becomes
    \begin{align*} \rho_{n+1}(u,d_\alpha) &\leq \rho_n(u,d_\alpha) - \frac{\nu_1}{2}\norm{x_n-x_{n+1}}^2 +\frac{\alpha^2\nu_1}{2}\|x_n-x_{n-1}\|^2 -\lambda_n\eps_n\|y_n-u\|^2 + \frac{2\lambda_n^2}{c_\alpha} \|\opD_{\eps_{n}}u+v\|^2  \nonumber  \\
    & \quad  + \frac{2\lambda_n^2\|p\|^2}{\nu_1}  + 2\lambda_n\beta_n\left[\varphi_{\opB_2}\bigg(u,\frac{p}{\beta_n}\bigg)-\sigma_\scrC\bigg(\frac{p}{\beta_n}\bigg)\right] + 2\lambda_n\inner{u-y_n,w+\eps_nu} ,
    \end{align*}
    which we rewrite as
    \begin{align} \label{eq:discrete3_case1}
    \rho_{n+1}(u,d_\alpha') &\leq \rho_n(u,d_\alpha') - \frac{\nu_1(1-\alpha^2)}{2}\norm{x_n-x_{n+1}}^2 -\lambda_n\eps_n\|y_n-u\|^2 + 2\lambda_n^2\left[\frac{\|\opD_{\eps_{n}}u+v\|^2}{c_\alpha} +\frac{\|p\|^2}{\nu_1}\right]  \nonumber  \\
    & \quad  + 2\lambda_n\beta_n\left[\varphi_{\opB_2}\bigg(u,\frac{p}{\beta_n}\bigg)-\sigma_\scrC\bigg(\frac{p}{\beta_n}\bigg)\right] + 2\lambda_n\inner{u-y_n,w+\eps_nu},  
    \end{align}
    where $d_\alpha':=d_\alpha+\frac{\alpha^2\nu_1}{2}$. \\
\noindent{\bf Case 2.} If $\opB_2\equiv 0$, we take $\gamma_2\in[0,1]$, and write
    \begin{align*}
        \delta_n & = 2\gamma_2\lambda_n\beta_n\inner{u-y_n,\opB_1 y_n} - 2\lambda_n\inner{u-y_n,p} +2(1-\gamma_2)\lambda_n\beta_n\inner{u-y_n,\opB_1 y_n}\\
        & = 2\gamma_2\lambda_n\beta_n\bigg[\inner{u-y_n,\opB_1 y_n}+\inner{y_n,\frac{p}{\gamma_2\beta_n}}-\inner{u,\frac{p}{\gamma_2\beta_n}}\bigg] -2\mu(1-\gamma_2)\lambda_n\beta_n\|\opB_1y_n\|^2 \\
        & \le 2\lambda_n\beta_n\left[\varphi_{\opB_1}\bigg(u,\frac{p}{\gamma_2\beta_n}\bigg)-\sigma_\scrC\bigg(\frac{p}{\gamma_2\beta_n}\bigg)\right]
         -2\mu(1-\gamma_2)\lambda_n\beta_n\|\opB_1y_n\|^2.
    \end{align*}
    In this case, \eqref{eq:discrete3} becomes
    \begin{align}\label{eq:discrete3_case2}
    \rho_{n+1}(u,d_\alpha) &\leq \rho_n(u,d_\alpha) - \nu_1\norm{x_n-x_{n+1}}^2 -\lambda_n\eps_n\|y_n-u\|^2
    + \frac{2\lambda_n^2}
    {c_\alpha}\|\opD_{\eps_{n}}u+v\|^2   \nonumber   \\
    & \quad + 2\lambda_n\beta_n\left[\frac{\lambda_n\beta_n}
    {c_\alpha}-\mu(1-\gamma_2)\right]\|\opB_1y_n\|^2 \nonumber \\
    & \quad + 2\lambda_n\beta_n\left[\varphi_{\opB_1}\bigg(u,\frac{p}{\gamma_2\beta_n}\bigg)-\sigma_\scrC\bigg(\frac{p}{\gamma_2\beta_n}\bigg)\right]  + 2\lambda_n\inner{u-y_n,w+\eps_nu}.
    \end{align}
\noindent{\bf Case 3.} If $\opB_1=\nabla \Psi_1$ and $\opB_2=\partial\Psi_2$, we have
    \begin{align*}
        \delta_n & = 2\lambda_n\beta_n\inner{u-y_n,\nabla \Psi_1(y_n)} + 2\lambda_n\beta_n\inner{u-x_{n+1},\xi_{n+1}}  - 2\lambda_n\inner{u-y_n,p} \\
        & = 2\gamma_2\lambda_n\beta_n\inner{u-y_n,\nabla \Psi_1(y_n)} + 2(1-\gamma_2)\lambda_n\beta_n\inner{u-y_n,\nabla \Psi_1(y_n)} \\
        & \quad + 2\lambda_n\beta_n\inner{u-x_{n+1},\xi_{n+1}}  - 2\lambda_n\inner{u-y_n,p}.
    \end{align*}
    Using Lemma \ref{L:transportation_formula} with $c=\mu$ on the first term on the right-hand side, we get
    \begin{align*}
        \delta_n & \le - 2\gamma_2\lambda_n\beta_n\Psi_1(x_{n+1})  - 2\lambda_n\beta_n\Psi_2(x_{n+1}) -2\lambda_n\inner{u-x_{n+1},p} - 2\lambda_n\inner{x_{n+1}-y_n,p}\\
        & \quad -2\mu\lambda_n\beta_n(1-\gamma_2)\|\nabla \Psi_1(y_n)\|^2 +\frac{2\gamma_2\lambda_n\beta_n}{\mu}\|x_{n+1}-y_n\|^2 \\
        & \le 2\gamma_2\lambda_n\beta_n\left[\langle x_{n+1},\frac{p}{\gamma_2\beta_n}\rangle - \big(\Psi_1+\Psi_2\big)(x_{n+1}) -\langle u,\frac{p}{\gamma_2\beta_n}\rangle\right] - 2(1-\gamma_2)\lambda_n\beta_n\Psi_2(x_{n+1})\\
        & \quad -2\mu\lambda_n\beta_n(1-\gamma_2)\|\nabla \Psi_1(y_n)\|^2 +\frac{2\gamma_2\lambda_n\beta_n}{\mu}\|x_{n+1}-y_n\|^2 + \frac{8\lambda_
        n^2\|p\|^2}{\nu_1} 
        + \frac{\nu_1}{8}\|x_{n+1}-y_n\|^2 \\
        & \le 2\gamma_2\lambda_n\beta_n\left[\big(\Psi_1+\Psi_2\big)^*(\frac{p}{\gamma_2\beta_n}) -\sigma_C(\frac{p}{\gamma_2\beta_n})\right] +\left[\frac{2\gamma_2\lambda_n\beta_n}{\mu}+\frac{\nu_1}{8}\right]\|x_{n+1}-y_n\|^2 + \frac{8\lambda_
        n^2\|p\|^2}{\nu_1} \\
        & \quad -2\mu(1-\gamma_2)\lambda_n\beta_n\|\nabla \Psi_1(y_n)\|^2-2(1-\gamma_2)\lambda_n\beta_n\Psi_2(x_{n+1}).
    \end{align*}
    Using this in \eqref{eq:discrete3}, we obtain
    \begin{align}\label{eq:discrete3_case3}
    \rho_{n+1}(u,d_\alpha) &\leq \rho_n(u,d_\alpha) - \nu_1\norm{x_n-x_{n+1}}^2 -\lambda_n\eps_n\|y_n-u\|^2
     \nonumber \\
    & \quad  + 2\lambda_n\beta_n\left[\frac{\lambda_n\beta_n}
    {c_\alpha}-\mu(1-\gamma_2)\right]\|\nabla \Psi_1(y_n)\|^2 -2(1-\gamma_2)\lambda_n\beta_n\Psi_2(x_{n+1})  \nonumber \\
    & \quad 
    + 2\lambda_n^2\left[\frac{\|\opD_{\eps_{n}}u+v\|^2}
    {c_\alpha} + \frac{4\|p\|^2}{\nu_1} \right] +\left[\frac{2\gamma_2\lambda_n\beta_n}{\mu}+\frac{\nu_1}{8}\right]\|x_{n+1}-y_n\|^2  \nonumber \\
    & \quad + 2\gamma_2\lambda_n\beta_n\left[\big(\Psi_1+\Psi_2\big)^*(\frac{p}{\gamma_2\beta_n}) -\sigma_\scrC(\frac{p}{\gamma_2\beta_n})\right] + 2\lambda_n\inner{u-y_n,w+\eps_nu}.
    \end{align}

The preceding discussion motivates the following: 

\begin{assumption} \label{A:standing_parameters}
$(\alpha_n)$ is nonincreasing and bounded from above by some $\alpha>0$. Additionally, the following bounds are satisfied:
$$\sup_{n\ge 0}\lambda_n\beta_n<c_\alpha\mu\footnote{If $\opB_1\equiv 0$, we interpret $\mu=\infty$.},\qquad \lambda_n\le\left[\frac{c_\alpha}{2(L+\eps_n)^2}\right]\eps_n,\qbox{where}c_\alpha<\frac{1-3\alpha}{2(1+\alpha^2)},$$ 
\end{assumption}

We set 
\begin{equation} \label{eq:kappa_uvp}
    \kappa(u,v,p):=2\sup_{n\ge 0}\left[\frac{\|\opD_{\eps_{n}}u+v\|^2}{c_\alpha} + \frac{4\|p\|^2}{\nu_1} \right],
\end{equation}
and note that 
$$\kappa(u,v,p)
\le \left[\frac{2\|\opD u+v\|^2}{c_\alpha} + \frac{2\|u\|^2}{c_\alpha}\sup_{n\ge 0}\eps_n^2 + \frac{4\|p\|^2}{\nu_1} \right]<\infty,$$
because $\eps_n\to 0$ as $n\to\infty$. 

\begin{proposition} \label{P:abstract_estimation}
Let $(x_n)_{n\geq 0}$ be the sequence generated by Algorithm \eqref{E:Algorithm4}, and let Assumption \ref{A:standing_parameters} hold. Consider a point $(u,w) \in \gr(\opA+\opD+\NC_{\scrC})$, so that $w=v+\opD u+p$ for some $v\in \opA u$ and $p\in \NC_{\scrC}(u)$. For every $n\geq 1$, we have
\begin{equation} \label{eq:discrete_all}
    \rho_{n+1}(u,d) + A_n + \lambda_n\eps_n\|y_n-u\|^2 \leq \rho_n(u,d) + \Delta_n(u,p) + \kappa(u,v,p)\lambda_n^2 + 2\lambda_n\inner{u-y_n,w+\eps_nu},
\end{equation}
where $A_n\ge \frac{4\nu_1}{9}\norm{x_n-x_{n+1}}^2$. The coefficients $\Delta_n(u,p)\ge 0$ and $d>0$ have concrete expressions for the cases 1-3. 
\end{proposition}

\begin{proof}
Since $\sup_{n\ge 0}\lambda_n\beta_n<c_\alpha\mu$, we can choose $\gamma_2\in(0,1]$ such that
$$\frac{\sup_{n\ge 0}\lambda_n\beta_n}{c_\alpha}<\mu(1-\gamma_2).$$
By further reducing $\gamma_2$, if necessary, we may assume that $\gamma_2 \le \frac{\nu_1}{16c_\alpha}$, which implies
$$\frac{2\gamma_2\lambda_n\beta_n}{\mu}\le 2\gamma_2c_\alpha \le \frac{\nu_1}{8},$$
for every $n\ge 0$. Defining
$$\nu_2:=2\left[\mu(1-\gamma_{2})-\frac{\sup_{n\ge 0}\lambda_n\beta_n}{c_\alpha}\right]>0,$$
we get
$$\frac{\lambda_n\beta_n}{c_\alpha}-\mu(1-\gamma_2)\le-\frac{\nu_2}{2}$$
for every $n\ge 0$.
\\

\noindent{\bf Case 1.} If $\opB_1\equiv 0$, inequality \eqref{eq:discrete3_case1} implies that
\begin{align} \label{eq:discrete3_case1'}
\rho_{n+1}(u,d_\alpha') &\leq \rho_n(u,d_\alpha') - \frac{4\nu_1}{9}\norm{x_n-x_{n+1}}^2 -\lambda_n\eps_n\|y_n-u\|^2 + \kappa(u,v,p)\lambda_n^2 \nonumber  \\
& \quad  + 2\lambda_n\beta_n\left[\varphi_{\opB_2}\bigg(u,\frac{p}{\beta_n}\bigg)-\sigma_\scrC\bigg(\frac{p}{\beta_n}\bigg)\right] + 2\lambda_n\inner{u-y_n,w+\eps_nu}.
\end{align}
Setting $A_n=\frac{4\nu_1}{9}\norm{x_n-x_{n+1}}^2$, $\Delta_n(u,p)=2\lambda_n\beta_n\left[\varphi_{\opB_2}\bigg(u,\frac{p}{\beta_n}\bigg)-\sigma_\scrC\bigg(\frac{p}{\beta_n}\bigg)\right]$ and $d=d_\alpha'$, we write \eqref{eq:discrete3_case1'} in the form of \eqref{eq:discrete_all}. \\
\noindent{\bf Case 2.} If $\opB_2\equiv 0$, \eqref{eq:discrete3_case2} gives
\begin{align} \label{eq:discrete3_case2'}
\rho_{n+1}(u,d_\alpha) &\leq \rho_n(u,d_\alpha) - \nu_1\norm{x_n-x_{n+1}}^2 -\lambda_n\eps_n\|y_n-u\|^2 - \nu_2\lambda_n\beta_n\|\opB_1y_n\|^2+\kappa(u,v,p)\lambda_n^2 \nonumber \\
& \quad + 2\lambda_n\beta_n\left[\varphi_{\opB_1}\bigg(u,\frac{p}{\gamma_2\beta_n}\bigg)-\sigma_\scrC\bigg(\frac{p}{\gamma_2\beta_n}\bigg)\right]  + 2\lambda_n\inner{u-y_n,w+\eps_nu}.
\end{align}
We write \eqref{eq:discrete3_case2'} in the form of \eqref{eq:discrete_all} by setting $A_n=\nu_1\norm{x_n-x_{n+1}}^2+\nu_2\lambda_n\beta_n\|\opB_1y_n\|^2$, $\Delta_n(u,p)=2\lambda_n\beta_n\left[\varphi_{\opB_1}\bigg(u,\frac{p}{\gamma_2\beta_n}\bigg)-\sigma_\scrC\bigg(\frac{p}{\gamma_2\beta_n}\bigg)\right]$ and $d=d_\alpha$. \\
\noindent{\bf Case 3.} If $\opB_1=\nabla \Psi_1$ and $\opB_2=\partial\Psi_2$, we deduce from \eqref{eq:discrete3_case3}, that
\begin{align*}
\rho_{n+1}(u,d_\alpha) &\leq \rho_n(u,d_\alpha) - \nu_1\norm{x_n-x_{n+1}}^2 - \lambda_n\eps_n\|y_n-u\|^2 -\nu_2\lambda_n\beta_n\|\nabla\Psi_1(y_n)\|^2 + \kappa(u,v,p)\lambda_n^2 \nonumber \\
& \quad -2(1-\gamma_2) \lambda_n\beta_n\Psi_2(x_{n+1})  + \frac{\nu_1}{4}\|x_{n+1}-y_n\|^2  \nonumber \\
& \quad + 2\gamma_2\lambda_n\beta_n\left[\big(\Psi_1+\Psi_2\big)^*(\frac{p}{\gamma_2\beta_n}) -\sigma_\scrC(\frac{p}{\gamma_2\beta_n})\right] + 2\lambda_n\inner{u-y_n,w+\eps_nu}.
\end{align*}
Since 
$$\frac{\nu_1}{4}\|x_{n+1}-y_n\|^2 \le \frac{\nu_1}{2}\|x_{n+1}-x_n\|^2 + \frac{\nu_1\alpha^2}{2}\|x_n-x_{n-1}\|^2 \le \frac{\nu_1}{2}\|x_{n+1}-x_n\|^2 + \frac{\nu_1}{18}\|x_n-x_{n-1}\|^2,$$
we can regroup the terms to obtain
\begin{align} \label{eq:discrete3_case3'}
\rho_{n+1}(u,d_\alpha') &\leq \rho_n(u,d_\alpha') - \frac{4\nu_1}{9}\norm{x_n-x_{n+1}}^2 -\nu_2\lambda_n\beta_n\|\nabla\Psi_1(y_n)\|^2 -2(1-\gamma_2) \lambda_n\beta_n\Psi_2(x_{n+1}) \nonumber \\
& \quad - \lambda_n\eps_n\|y_n-u\|^2  + \kappa(u,v,p)\lambda_n^2 \nonumber \\
& \quad + 2\gamma_2\lambda_n\beta_n\left[\big(\Psi_1+\Psi_2\big)^*(\frac{p}{\gamma_2\beta_n}) -\sigma_\scrC(\frac{p}{\gamma_2\beta_n})\right] + 2\lambda_n\inner{u-y_n,w+\eps_nu}
\end{align}
Finally, we can set $A_n=\frac{4\nu_1}{9}\norm{x_n-x_{n+1}}^2 +\nu_2\lambda_n\beta_n\|\nabla\Psi_1(y_n)\|^2 + 2(1-\gamma_2) \lambda_n\beta_n\Psi_2(x_{n+1})$, $\Delta_n(u,p)=2\gamma_2\lambda_n\beta_n\left[\big(\Psi_1+\Psi_2\big)^*(\frac{p}{\gamma_2\beta_n}) -\sigma_\scrC(\frac{p}{\gamma_2\beta_n})\right]$ and $d=d_\alpha'$, to write \eqref{eq:discrete3_case3'} in the form of \eqref{eq:discrete_all}. 
\end{proof}

\subsection{Weak ergodic convergence}

%
%


Endowed with the preliminary energetic bounds from the previous section, we can now state and proof the announced weak convergence results. The following assumption will turn out to be crucial. 

\begin{assumption}[Attouch-Czarnecki condition] \label{A:Attouch_Czarnecki} 
The sequence $(\lambda_n\eps_n)_{n}$ belongs to $\ell^1_{+}(\N)$, and one of the following conditions holds:
\begin{itemize}
    \item $\opB_1\equiv 0$ and 
    $\displaystyle\sum_{n\ge 0}\lambda_n\beta_n\left[\varphi_{\opB_2}\bigg(u,\frac{p}{\beta_n}\bigg)-\sigma_\scrC\bigg(\frac{p}{\beta_n}\bigg)\right]<\infty$;
    \item $\opB_2\equiv 0$ and $\displaystyle \sum_{n\ge 0}\lambda_n\beta_n\left[\varphi_{\opB_1}\bigg(u,\frac{p}{\beta_n}\bigg)-\sigma_\scrC\bigg(\frac{p}{\beta_n}\bigg)\right]<\infty$; or
    \item $\opB_1=\nabla \Psi_1$, $\opB_2=\partial \Psi_2$ and 
    $\displaystyle \sum_{n\ge 0}\lambda_n\beta_n\left[(\Psi_1+\Psi_2)^*\bigg(\frac{p}{\beta_n}\bigg)-\sigma_\scrC\bigg(\frac{p}{\beta_n}\bigg)\right]<\infty$.
\end{itemize}
\end{assumption}
\begin{remark}
    Albeit abstract looking, the conditions formulated in Assumption \ref{A:Attouch_Czarnecki} define a useful geometric setting in which iterative regularization methods are powerful. To give a concrete setting the integrability conditions are satisfied, let us focus on the last case in which the penality operators are differential operators of convex functions. Consider the setting $\nabla \Psi:=\opB_{1}:\scrH\to\bar{\R}:=\R\cup\{-\infty,\infty\}$ be the proper convex lower semicontinuous function describing the penalty for the constraint domain $\scrC=\argmin(\Psi)$. The function $\Psi$ satisfies an Hölderian growth condition with exponent $\rho\geq 1$ if for some $\tau>0$ we have 
$$
\frac{\tau}{\rho}\dist(x,\argmin\Psi)^{\rho}\leq \Psi(x)-\min\Psi(x)\qquad\forall x\in\scrH.
$$
The Attouch-Czarnecki condition in that setting is known to be weaker than the Hölderian growth condition, since 
$$
0\leq \Psi^{*}(z)-\sigma_{\scrC}(p)\leq \frac{\tau^{1-\rho^{*}}}{\rho^{*}}\norm{z}^{\rho^{*}}\qquad\forall z\in\scrH, 
$$
where $\rho^{*}=\frac{\rho}{\rho-1}$. Hence, the Attouch-Czarnecki condition will hold automatically, whenever $\sum_{n\geq 0}\lambda\beta_{n}^{1-\rho^{*}}<\infty$. 
\end{remark}

With this remark in mind, we start developing the formal argument. 
\begin{proposition}\label{P:weak_convergence}
Let Assumptions \ref{A:standing_parameters} and \ref{A:Attouch_Czarnecki} hold, and let $(x_n,y_n)_{n\geq 0}$ be generated by \eqref{E:Algorithm4}. Then, the series $\sum_{n}\norm{x_{n+1}-x_{n}}^2$ is convergent, and, for every $u\in \zer(\opA+\opD+\NC_{\scrC})$, the limit $\lim\limits_{n\rightarrow+\infty}\norm{x_n-u}^2$ exists.
\end{proposition}

\begin{proof}
Take $u\in \zer(\opA+\opD+\NC_{\scrC})$, and use $w=0$ in Proposition \ref{P:abstract_estimation}. Then use Young's inequality on the last term to deduce that
$$\rho_{n+1}(u,d) + A_n + \lambda_n\eps_n\|y_n-u\|^2 \leq \rho_n(u,d) + \Delta_n(u,p) + \kappa(u,v,p)\lambda_n^2 + 2\lambda_n\eps_n\inner{u-y_n,u},$$
which implies
\begin{align} \label{eq:ell2_estimation}
\rho_{n+1}(u,d) + A_n & \leq \rho_n(u,d) + \Delta_n(u,p) + \kappa(u,v,p)\lambda_n^2 + \lambda_n\eps_n\|u\|^2 \nonumber \\
& \le \rho_n(u,d) + \Delta_n(u,p) + \left[\frac{\kappa(u,v,p) c_\alpha}{2L^2}+\|u\|^2\right]\lambda_n\eps_n.    
\end{align}

We use Lemma \ref{L:real1} with $\rho_n=\rho_n(u,d)$, $\omega_n=\|x_n-u\|^2$, $e_n=A_n$, $d_n=\Delta_n(u,p)+\left[\frac{\kappa(u,v,p) c_\alpha}{2L^2}+\|u\|^2\right]\lambda_n\eps_n$ (which is summable, in view of Assumption \ref{A:Attouch_Czarnecki}), $a_n=\alpha_n$ and $b_n=d\|x_{n+1}-x_n\|^2$. From part i), we conclude that $A_n\in\ell^1_{+}(\N)$, that $\lim_{n\to\infty}b_n=0$ (because $b_n$ can be bounded from above by a multiple of $A_n$),
and that $\rho:=\lim_{n\to\infty} \rho_n(u,d)$ exists. Since $(\alpha_n)$ is non-decreasing and bounded from above by $\frac{1}{3}$, and since $\lim_{n\to\infty}b_n=0$, part ii) shows that $\lim_{n\to\infty}\|x_n-u\|^2$ exists.
\end{proof}

\begin{remark}
By Proposition \ref{P:weak_convergence}, we have $0\le \lim_{n\to\infty}\|y_n-x_n\|\le\alpha\lim_{n\to\infty}\|x_{n+1}-x_n\|=0$. As consequence, the sequences $(x_n)_{n\geq 0}$ and $(y_n)_{n\geq 0}$ have the same weak limit points.
\end{remark}

\begin{theorem}
Let Assumptions \ref{A:standing_parameters} and \ref{A:Attouch_Czarnecki} hold, and let $(\lambda_n)\notin\ell^1_{+}(\N)$. Consider a sequence $(x_n,y_n)_{n\geq 0}$, generated by \eqref{E:Algorithm4}, and define the average sequence $(z_n)_{n\geq 1}$ by
$$z_n:=\frac{1}{\tau_n}\sum\limits^{n}_{k=1}\lambda_{k}y_{k}, \quad \text{where }\tau_{n}:=\sum\limits^{n}_{k=1}\lambda_{k}.$$
If $\opA+\opD+\NC_{\scrC}$ is maximally monotone and $\zer(\opA+\opD+\NC_{\scrC})\neq\emptyset$, then $z_n$ converges weakly, as $n \to \infty$, to a point in $\zer(\opA+\opD+\NC_{\scrC})$.
\end{theorem}

\begin{proof}
From Proposition \ref{P:weak_convergence} and Opial's Lemma (Lemma \ref{lem:Opial}), it suffices to prove that every weak limit point of $(z_n)$ belongs to $\zer(\opA+\opD+\NC_{\scrC})$. Take any $(u,w)\in \gr(\opA+\opD+\NC_{\scrC})$, and assume $z_{N_m}\rightharpoonup \bar z$, as $m\to\infty$. By \eqref{eq:ell2_estimation}, we have
$$\rho_{n+1}(u,d) \le \rho_n(u,d) + \Delta_n + c\lambda_n\eps_n,$$
with $c:=\left[\frac{\kappa(u,v,p) c_\alpha}{2L^2}+\|u\|^2\right]$. Summing for $n=1,\dots,N_m$, and dividing by $\tau_{N_m}$, we obtain
$$0\le \frac{1}{\tau_{N_m}}\rho_{N_m+1}(u,d) \leq \frac{1}{\tau_{N_m}}\left[\rho_1(u,d)+\sum_{n=1}^{N_m}\Delta_n+c\sum_{n=1}^{N_m}\lambda_n\eps_n\right]+2\inner{u-z_{N_m},w}.
$$ 
Letting $m\to\infty$, it ensues that
$$\inner{u-\bar z,w}\ge 0.$$
Since this holds for any $(u,w)\in \gr(\Phi)$, we conclude that $\bar z\in \zer(\Phi)$, by maximality.
\end{proof}





\section{Central path and the strong convergence to the least norm solution}
\label{strongconvergence}

So far, we have established only weak convergence of the ergodic trajectory. In this section we establish strong convergence to the least norm solution of the constrained variational inequality \eqref{eq:MI}. We thereby generalize our previous results \cite{peypouquet2026asymptotic} to inertial forward-backward type schemes with multiscale effects. Our proof relies on a-priori estimates of the central path $u^{\eps}$, defined as the unique solution to the regularized inclusion 
\begin{equation}\label{centralpath}
0\in (\opA+\opD+\NC_{\scrC}+\eps\Id_{\scrH})u^{\eps}\iff -u^{\eps}\in\Phi u^{\eps}. 
\end{equation}
The Tikhonov regularized operator is defined as $\Phi_{\eps}:=\Phi+\eps_n\Id_{\scrH}.$
Clearly, $\{\Phi_{\eps}\}_{\eps>0}$ is a family of strongly monotone set-valued operators, hence $u^{\eps}$ is uniquely defined by \eqref{centralpath}. Note that the penalty operator $\beta_n\opB$ is not absorbed into the central path, which instead encodes the constraint through the normal cone directly. It is instead handled at the algorithmic level via the Fitzpatrick function, whose summability against the parameter sequence $(\lambda_n\beta_n)_{n}$ formulated in a Attouch-Czarnecki hypothesis (cf. Assumption \ref{A:Attouch_Czarnecki_strong}), controls the discrepancy between $\beta_n\opB(\cdot)$ and $\NC_{\scrC}(u^{\eps})$ in the energy inequality. Compared to the double penalization central path discussed in \cite{peypouquet2026asymptotic}, three structural advantages motivate this shift. 
\begin{itemize}
\item[(i)] The path $(u^{\eps_{n}})_{n}$ depends only on the limiting inclusion and the Tikhonov parameter $\eps_n$, and is therefore completely decoupled from the penalty schedule $\beta_n$. In contrast to the double-penalization path $u^{\eps,\beta}$ used in \cite{peypouquet2026asymptotic} as reference, its definition is intrinsic to the limiting problem and does not depend on the penalization parameters.
\item[(ii)] Compared to \cite{peypouquet2026asymptotic}, the regularity analysis is considerably simplified (cf. Section 2 of that paper). The variation of the path depends only on the slow evolution of the Tikhonov parameter $\eps_n$, while the boundedness of the associated normal-cone selections follows from a pointwise qualification condition and the classical local boundedness of maximal monotone operators. This contrasts with the two-parameter approach, where uniform estimates are required along the whole central path together with additional controls on the penalization components.
\item[(iii)] The path converges to the least-norm solution selected by the limiting problem. Moreover, the Attouch–Czarnecki condition acts directly on the penalty operators and simultaneously accounts for the two penalization terms. As a consequence, no additional asymptotic exact-penalty assumption is required.
\end{itemize}

\subsection{Perturbed solutions and the central funnel}


Given a positive sequence $(\eps_{n})_{n\in\mathbb{N}}$ satisfying $\eps_n\to0$, the sequence $(u^{\eps_{n}})_{n\in\mathbb{N}}$ generated by \eqref{centralpath} is called the discrete Tikhonov central path. We shall study its asymptotic properties as $n \to \infty$ and use it as a reference path in the convergence analysis of the proposed algorithm.

\begin{lemma}\label{lem:proj}
$\scrS_{\eps}:=\Phi_{\eps}^{-1}(0)$ is non-empty, closed and convex, hence $u^*=\argmin\{\norm{x}:x\in\scrS\}=\operatorname{Proj}_{\scrS}(0)$ exists, is unique, and is characterized by
\begin{equation}\label{eq:varchar} 
\inner{u^*,z-u^*}\ge 0,\qquad\forall z\in \scrS_{\eps}.
\end{equation}
\end{lemma}
\begin{proof}
Maximal monotonicity of $\Phi_{\eps}$ implies that $\scrS_{\eps}$ is closed and convex. Combined with $\scrS_{\eps}\neq\emptyset$, the Hilbert-space projection theorem yields the existence, uniqueness, and the variational characterization \eqref{eq:varchar}.
\end{proof}

The following lemma collects the classical properties of the Tikhonov regularization, as in \cite{Bot:2020aa}.

\begin{lemma}\label{lem:tikhonov}
For each $\eps>0$, the inclusion $0\in (\Phi+\eps \Id)u$ has a unique solution $u^{\eps}$. Moreover
\begin{itemize}
\item[(i)]  $\norm{u^{\eps}}\le\norm{u^*}$ for every $\eps>0$;
\item[(ii)] $u^{\eps}\to u^*$ strongly as $\eps\to 0^{+}$;
\item[(iii)] for $\eps_1>\eps_2>0$,
\begin{equation}\label{eq:bvtikh}
\norm{u^{\eps_1}-u^{\eps_2}}\le \norm{u^*}\,\frac{\eps_1-\eps_2}{\eps_2}.
\end{equation}
\end{itemize}
\end{lemma}
\begin{proof}
$\Phi_{\eps}:=\Phi+\eps \Id$ is maximally monotone and $\eps$-strongly monotone, so $u^\eps$ exists and is unique.

\emph{(i)} For any $z\in \scrS$, $0\in \Phi_{\eps} z$ and $-\eps u^\eps\in \Phi u^\eps$. Monotonicity gives
$$
\inner{-\eps u^\eps-0,u^\eps-z}\ge 0
\Longrightarrow \norm{u^\eps}^{2}\le\inner{u^\eps,z}
\le\norm{u^\eps}\,\norm{z},
$$
hence $\norm{u^\eps}\le\norm{z}$. Take $z=u^*$.

\emph{(ii)} By (i) the family is bounded; let $\bar u$ be a weak cluster point along $\eps_n\to 0$. For $(z,w)\in\gr \Phi$,
$$
\inner{-\eps_n u^{\eps_n}-w,u^{\eps_n}-z}\ge 0 \Longrightarrow \inner{-w,u^{\eps_n}-z}\ge \eps_n \inner{u^{\eps_n},u^{\eps_n}-z}.
$$
Letting $n\to\infty$, the right-hand side tends to $0$, so $\inner{w,z-\bar u}\ge 0$ for every $(z,w)\in\gr \Phi$. Maximal monotonicity forces $0\in \Phi \bar u$, i.e.\ $\bar u\in \scrS$. Combined with $\norm{\bar u}\le\liminf\norm{u^{\eps_n}}\le\norm{u^*}$ and the uniqueness of the minimum-norm element, $\bar u=u^*$. The whole net thus converges weakly to $u^*$, and $\limsup\norm{u^\eps}\le\norm{u^*}\le\liminf\norm{u^\eps}$ yields strong convergence.

\emph{(iii)} Apply monotonicity of $\Phi$ to the points $-\eps_1 u_1\in \Phi u_1$ and $-\eps_2 u_2\in \Phi u_2$:
$$
\inner{-\eps_1 u_1+\eps_2 u_2,u_1-u_2}\ge 0
\;\Longrightarrow\;
\eps_2\norm{u_1-u_2}^{2}\le (\eps_1-\eps_2)\inner{u_2,u_2-u_1}\le (\eps_1-\eps_2)\norm{u_2}\cdot\norm{u_1-u_2}.
$$
Combined with $\norm{u_2}\le\norm{u^*}$ this gives \eqref{eq:bvtikh}.
\end{proof}

\subsection{Selection along the central path and boundedness of accompanying normal cone elements}
Let $(\eps_{n})_{n}$ be a positive real-valued sequence with $\eps_{n}\to 0$. Set $u_n:=u^{\eps_n}$. By definition,
\begin{equation}\label{eq:undef}
0\in (\opA+\opD+\eps_n \Id_{\scrH}+\NC_{\scrC})u_n,\qquad
u_n\in\dom \Phi_{\eps}\subset\scrC.
\end{equation}
There exist selections $v_n\in\opA u_n$, $p_n\in\NC_{\scrC}u_n$ with
\begin{equation}\label{eq:undecomp}
0=v_n+\opD u_n+\eps_n u_n+p_n.
\end{equation}
The decomposition is not unique. However, as shown in the next Lemma, we can perform selections of the accompanying normal cone elements $(p_{n})_{n}$ which we can take as a bounded sequence, under the following assumption. 
\begin{assumption}\label{ass:Ainterior}
    The unique least norm solution $u^{*}=\operatorname{Proj}_{\scrS}(0)$ lies in the interior of the domain of $\opA$: $u^{*}\in\Int\dom\opA$. In particular, $\opA$ has a non-empty interior. 
\end{assumption}
 \begin{remark}
 The boundedness of $p_n$ established below does not depend on the specific choice of $v_n\in\opA u_n$.
\end{remark}
\begin{remark}
    Recently, the convergence of the Tikhonov trajectory to the least norm solution of splitting problems with operators having empty interiors has been established in \cite{garrido2026stochastic}. Their analysis is confined to finite-dimensional inclusions, hence weak convergence coincides with strong convergence. It is, however, likely that similar techniques can be applied to the present setting of this paper. We leave this technically interesting question for future research. 
\end{remark}
\begin{remark}
    In infinite dimensional Hilbert spaces, the non-empty interior assumption is of course quite demanding. It is, however, not too uncommon in splitting dynamics with obstacle terms (i.e. the hard constraint $\scrC$); see e.g. \cite{rascanu2014deterministic}. The technique described in \cite{garrido2026stochastic} is a potential avenue for relaxing this assumption. We leave this question for future research.
\end{remark}

\begin{lemma}\label{lem:pn}
There exists $N_{0}\ge 1$ such that
\begin{equation}\label{eq:K}
\kappa\eqdef\sup_{n\ge N_{0}}\norm{p_n}\;<\;+\infty
\end{equation}
for any selection $(v_n,p_n)$ satisfying \eqref{eq:undecomp}.
\end{lemma}

\begin{proof}
By \cite{BauCom16}, the maximally monotone operator $\opA$ is locally bounded at every interior point of its domain: since $u^*\in\Int\dom\opA$. There exist $r,M>0$ such that
$$
\opA(u)\subset B(0,M)\quad\text{(as a set)} \quad\text{for every }u\in B(u^{*},r).
$$
By Lemma~\ref{lem:tikhonov}\textup{(ii)}, $u_n\to u^{*}$, so there exists $N_{0}\ge 1$ with $u_n\in B(u^{*},r)$ for all $n\ge N_{0}$. In particular, \emph{any} selection
$v_n\in\opA u_n$ satisfies $\norm{v_n}\le M$. From \eqref{eq:undecomp},
$$
\norm{p_n}\le\norm{v_n}+\norm{\opD u_n}+\eps_n\norm{u_n}.
$$
Since $\opD$ is $\frac{1}{\eta}$-Lipschitz continuous, we have
$$
\norm{\opD x - \opD y} \le \eta^{-1} \norm{x - y}, \quad \forall x,y.
$$
Taking $y = u^*$, we obtain
$$
\norm{\opD u_n} \le \norm{\opD u^*} + \norm{\opD u_n - \opD u^*} \le \norm{\opD u^*} + \eta^{-1} \norm{u_n - u^*}.
$$
 Since $u_n \to u^*$, the sequence $\|u_n - u^*\|$ is bounded (and even converges to $0$). Hence, there exists a constant $C > 0$ such that
$$
\norm{\opD u_n} \le C,
$$
which shows that $(\opD u_n)_n$ is bounded. By Lemma~\ref{lem:tikhonov}\textup{(i)}, $\eps_n\norm{u_n}\le\eps_n\norm{u^{*}}$. Combining these three bounds, $\kappa=\sup_{n\ge N_{0}}\norm{p_n}<+\infty$.
\end{proof}

\subsection{Strong convergence to the minimum-norm solution}

In what follows, we assume Assumption \ref{A:standing_parameters} being in place. For each $n\ge 0$, denote by $u_n$ the unique zero of the strongly monotone operator $\opA+\opD_{\eps_n}+\NC_{\scrC}$. In other words,
$$0\in \opA u_n+\opD u_n+\eps_nu_n+\NC_{\scrC}(u_n),$$
which means that there exist $v_n\in \opA u_n$ and $p_n\in \NC_{\scrC}(u_n)$ such that
$$v_n+\opD u_n+\eps_nu_n+p_n=0.$$
At the same time, $(u_n,-\eps_nu_n)\in \gr(\opA+\opD_{\eps_n}+\NC_{\scrC})$ for all $n$. Using Proposition \ref{P:abstract_estimation}, we get
\begin{equation} \label{eq:discrete_strong}
\rho_{n+1}(u_n,d) - \rho_n(u_n,d) + \frac{4\nu_1}{9}\|x_{n+1}-x_n\|^2 + \lambda_n\eps_n\|y_n-u_n\|^2 \leq  \Delta_n(u_n,p_n) + \kappa(u_n,v_n,p_n)\lambda_n^2,
\end{equation}
where $\nu_1>0$ and $\rho_n$ is defined by \eqref{eq:rho_n}. The first two terms on the left-hand side of \eqref{eq:discrete_strong} are
\begin{align*}
\rho_{n+1}(u_n,d) - \rho_n(u_n,d)  
& = \norm{x_{n+1}-u_n}^2-\alpha_n\norm{u_n-x_n}^2+d\norm{x_{n+1}-x_n}^2 \\
& \quad - \norm{x_n-u_n}^2 +\alpha_{n-1}\norm{x_{n-1}-u_n}^2 - d\norm{x_n-x_{n-1}}^2 \\
& = \norm{x_{n+1}-u_n}^2 + \alpha_{n-1}\norm{x_{n-1}-u_n}^2 
-(1+\alpha_n)\norm{x_n-u_n}^2 \\
& \quad +d\big[\norm{x_{n+1}-x_n}^2-\norm{x_n-x_{n-1}}^2\big].
\end{align*}
Take a positive sequence $(\zeta_n)_{n}$, to be specified later. By Young's inequality, we have
\begin{align*}
\norm{x_{n+1}-u_n}^2 & \ge (1-\zeta_n)\|x_{n+1}-u_{n+1}\|^2+(1-\zeta_n^{-1})\norm{u_{n+1}-u_n}^2  \\
\alpha_{n-1}\norm{x_{n-1}-u_n}^2 & \ge \alpha_{n-1}(1-\zeta_{n-1})\norm{x_{n-1}-u_{n-1}}^2+\alpha_{n-1}(1-\zeta_{n-1}^{-1})\norm{u_{n-1}-u_n}^2.
\end{align*}
It follows that
\begin{align*}
\rho_{n+1}(u_n,d) - \rho_n(u_n,d) 
& \ge (1-\zeta_n)\|x_{n+1}-u_{n+1}\|^2+(1-\zeta_n^{-1})\norm{u_{n+1}-u_n}^2  \\
& \quad +\alpha_{n-1}(1-\zeta_{n-1})\norm{x_{n-1}-u_{n-1}}^2+\alpha_{n-1}(1-\zeta_{n-1}^{-1})\norm{u_{n-1}-u_n}^2 \\
& \quad -(1+\alpha_n)\norm{x_n-u_n}^2+d\big[\norm{x_{n+1}-x_n}^2-\norm{x_n-x_{n-1}}^2\big].
\end{align*}
For the last term on the left-hand side of \eqref{eq:discrete_strong}, use Young's inequality once more to get
\begin{align*}
\lambda_n\eps_n\|y_n-u_n\|^2 & = 
\lambda_n\eps_n\|x_n-u_n+\alpha_n(x_n-x_{n-1})\|^2\\
& \ge \frac{\lambda_n\eps_n}{2}\|x_n-u_n\|^2-\alpha_n^2\lambda_n\eps_n\|x_n-x_{n-1}\|^2.
\end{align*}
Back to \eqref{eq:discrete_strong}, this entails
\begin{align*}
\Delta_n(u_n,p_n) + \kappa(u_n,v_n,p_n)\lambda_n^2 
& \ge \rho_{n+1}(u_n,d) - \rho_n(u_n,d) + \frac{4\nu_1}{9}\norm{x_n-x_{n+1}}^2 + \lambda_n\eps_n\|y_n-u_n\|^2 \\
& \ge (1-\zeta_n)\|x_{n+1}-u_{n+1}\|^2+(1-\zeta_n^{-1})\norm{u_{n+1}-u_n}^2  \\
& \quad +\alpha_{n-1}(1-\zeta_{n-1})\norm{x_{n-1}-u_{n-1}}^2+\alpha_{n-1}(1-\zeta_{n-1}^{-1})\norm{u_{n-1}-u_n}^2 \\
& \quad -(1+\alpha_n)\norm{x_n-u_n}^2 +d\big[\norm{x_{n+1}-x_n}^2-\norm{x_n-x_{n-1}}^2\big]\\
& \quad + \frac{\lambda_n\eps_n}{2}\|x_n-u_n\|^2 + \frac{4\nu_1}{9}\norm{x_n-x_{n+1}}^2 -\alpha_n^2\lambda_n\eps_n\|x_n-x_{n-1}\|^2 \\
& = (1-\zeta_n)\big[\|x_{n+1}-u_{n+1}\|^2-\alpha_n\|x_n-u_n\|^2\big] + \big(\alpha_{n+1}^2\lambda_{n+1}\eps_{n+1}+d\big)\norm{x_{n+1}-x_n}^2\\
& \quad - (1-\zeta_{n-1})\big[\|x_n-u_n\|^2-\alpha_{n-1}\|x_{n-1}-u_{n-1}\|^2\big] -\big(\alpha_n^2\lambda_n\eps_n+d\big)\norm{x_n-x_{n-1}}^2 \\
& \quad +\left[\alpha_n(1-\zeta_n)+(1-\zeta_{n-1})-(1+\alpha_n)+\frac{\lambda_n\eps_n}{2}\right]\|x_n-u_n\|^2\\
& \quad +(1-\zeta_n^{-1})\norm{u_{n+1}-u_n}^2 +\alpha_{n-1}(1-\zeta_{n-1}^{-1})\norm{u_{n-1}-u_n}^2 \\
& \quad + \left(\frac{4\nu_1}{9}-\alpha_{n+1}^2\lambda_{n+1}\eps_{n+1}\right)\norm{x_{n+1}-x_n}^2.
\end{align*}
Writing 
\begin{equation} \label{eq:definition_varrho}
    \varrho_n:=(1-\zeta_{n-1})\big[\|x_n-u_n\|^2-\alpha_{n-1}\|x_{n-1}-u_{n-1}\|^2\big]
    +\big(\alpha_n^2\lambda_n\eps_n+d\big)\norm{x_n-x_{n-1}}^2,
\end{equation}
simplifying, and discarding some of the positive terms, we obtain
\begin{align*}
\Delta_n(u_n,p_n) + \kappa(u_n,v_n,p_n)\lambda_n^2
& \ge \varrho_{n+1}-\varrho_n  +\left[\frac{\lambda_n\eps_n}{2}-\alpha\zeta_n-\zeta_{n-1}\right]\|x_n-u_n\|^2\\
& \quad -\zeta_n^{-1}\norm{u_{n+1}-u_n}^2 -\alpha\zeta_{n-1}^{-1}\norm{u_{n-1}-u_n}^2  + \left(\frac{4\nu_1}{9}-\alpha^2\lambda_{n+1}\eps_{n+1}\right)\norm{x_{n+1}-x_n}^2.
\end{align*}
For simplicity, we assume that the sequence $(\lambda_n\eps_n)$ is nonincreasing, and define 
$$\zeta_n:=\frac{\lambda_{n+1}\eps_{n+1}}{3(1+\alpha)}.$$
Then, 
$$\frac{\lambda_n\eps_n}{2}-\alpha\zeta_n-\zeta_{n-1} \ge \frac{\lambda_n\eps_n}{2}-(1+\alpha)\zeta_{n-1}=\frac{\lambda_n\eps_n}{2}-\frac{\lambda_n\eps_n}{3}=\frac{\lambda_n\eps_n}{6},$$
and so
$$\varrho_{n+1}-\varrho_n +\frac{\lambda_n\eps_n}{6}\|x_n-u_n\|^2 + \left(\frac{4\nu_1}{9}-\alpha^2\lambda_{n+1}\eps_{n+1}\right)\norm{x_{n+1}-x_n}^2
\le \Delta_n(u_n,p_n) + \kappa(u_n,v_n,p_n)\lambda_n^2 + r_n +r_{n-1},
$$
where we have written
$$r_n:=\frac{3(1+\alpha)\norm{u_{n+1}-u_n}^2}{\lambda_{n+1}\eps_{n+1}}.$$
Since $\lambda_n\eps_n\to 0$ 
, we may assume, without loss of generality, that $\alpha^2\lambda_{n+1}\eps_{n+1}\le\frac{\nu_1}{9}$. Therefore,
\begin{equation} \label{eq:varrho_strong}
\varrho_{n+1}-\varrho_n + \frac{\nu_1}{3}\norm{x_{n+1}-x_n}^2 +\frac{\lambda_n\eps_n}{6}\|x_n-u_n\|^2 \le \Delta_n(u_n,p_n) + \kappa(u_n,v_n,p_n)\lambda_n^2 + r_n + r_{n-1}.
\end{equation}

\begin{assumption} \label{A:Attouch_Czarnecki_strong}
    \begin{itemize}
        \item [i)] $(\lambda_n)\in\ell^2_{+}(\N)$ and satisfies slow decay: 
        $$\sum_{n\ge 0}\frac{(\eps_{n+1}-\eps_n)^2}{\lambda_{n+1}\eps_{n+1}^3}.$$
        \item [ii)] (locally uniform Attouch-Czarnecki condition) One of the following conditions holds:
    \begin{itemize}
    \item $\opB_1\equiv 0$ and 
    $$\sum_{n\ge 0}\lambda_n\beta_n\left[\varphi_{\opB_2}\bigg(u_n,\frac{p_n}{\beta_n}\bigg)-\sigma_\scrC\bigg(\frac{p_n}{\beta_n}\bigg)\right]<\infty;$$
    \item $\opB_2\equiv 0$ and 
    $$\sum_{n\ge 0}\lambda_n\beta_n\left[\varphi_{\opB_1}\bigg(u_n,\frac{p_n}{\beta_n}\bigg)-\sigma_\scrC\bigg(\frac{p_n}{\beta_n}\bigg)\right]<\infty;$$
    \item $\opB_1=\nabla \Psi_1$, $\opB_2=\partial \Psi_2$ and 
    $$\sum_{n\ge 0}\lambda_n\beta_n\left[(\Psi_1+\Psi_2)^*\bigg(\frac{p_n}{\beta_n}\bigg)-\sigma_\scrC\bigg(\frac{p_n}{\beta_n}\bigg)\right]<\infty.$$
    \end{itemize}
    \end{itemize}
\end{assumption}

\begin{theorem}\label{th:strongconvergence_main}
Let Assumptions \ref{A:standing_parameters}, \ref{ass:Ainterior} and \ref{A:Attouch_Czarnecki_strong} hold. Let $(x_n,y_n)_{n\geq 0}$ be generated by \eqref{E:Algorithm4}. Then, both $x_n$ and $y_n$ converge strongly, as $n \to \infty$, to the least-norm element of $\zer(\opA+\opD+\NC_{\scrC})$.
\end{theorem}

\begin{proof}
Define
\begin{align*}
    \omega_n & := (1-\zeta_{n-1})\|x_n-u_n\|^2 \\
    a_n & := \frac{\alpha_n}{(1-\zeta_{n-1})} \\
    b_n & := \big(\alpha_n^2\lambda_n\eps_n+d\big)\norm{x_n-x_{n-1}}^2 \\
    \rho_n & := \varrho_n \\
    d_n & := \Delta_n(u_n,p_n) + \kappa(u_n,v_n,p_n)\lambda_n^2 + r_n + r_{n-1} \\
    e_n & := \frac{\nu_1}{3}\norm{x_{n+1}-x_n}^2 + \frac{\lambda_n\eps_n}{6}\|x_n-u_n\|^2.
\end{align*}
In view of \eqref{eq:definition_varrho} and \eqref{eq:varrho_strong}, if $\sum d_n<\infty$, Lemma \ref{L:real1} shows that 
$$\lim_{n\to\infty}\|x_n-u_n\|=\lim_{n\to\infty}\omega_n$$
exists, and that $$\sum_{n\ge 0}\lambda_n\eps_n\|x_n-u_n\|^2<\infty.$$
Since $\sum\lambda_n\eps_n=\infty$, necessarily $\lim_{n\to\infty}\|x_n-u_n\|=0$. Since $u_n$ converges strongly to $u^*$, so do $x_n$ and $y_n$. It only remains to prove that $\sum d_n<\infty$. Combining part i) of Assumption \ref{A:Attouch_Czarnecki_strong} with Lemma \ref{lem:pn}, we see that $v_n$ and $p_n$ are bounded, which implies that
$\mathcal K:=\sup_{n\ge 0}\kappa(u_n,v_n,p_n)<\infty$. It follows that $\sum\kappa(u_n,v_n,p_n)\lambda_n^2<\infty$. \\
Inequality \eqref{eq:bvtikh} of Lemma \ref{lem:tikhonov} gives
$$r_n=\frac{3(1+\alpha)\norm{u_{n+1}-u_n}^2}{\lambda_{n+1}\eps_{n+1}}\le \frac{3(1+\alpha)\norm{u^*}^2}{\lambda_{n+1}\eps_{n+1}}\left(\frac{\eps_n-\eps_{n+1}}{\eps_{n+1}}\right)^2.$$
By part ii) of Assumption \ref{A:Attouch_Czarnecki_strong}, $\sum r_n<\infty$. Finally, part iii) of Assumption \ref{A:Attouch_Czarnecki_strong} gives the summability of $ \Delta_n(u_n,p_n)$, and allows us to conclude.
\end{proof}

\section{Numerical Experiments}
\label{sec:numerics}
In this section, we present numerical experiments to demonstrate the applicability of the proposed numerical scheme. To evaluate the effect of inertia, we compare the inertial FB dynamical system algorithm \ref{E:Algorithm4} with its non-inertial counterpart (i.e., $\alpha_n \equiv 0$), analyzed in \cite{peypouquet2026asymptotic}.

\subsection{Problem description}
We consider the image inpainting problem from \cite{cortild2025regularization}. We represent a grayscale image $Y$ of dimension
$(P, Q)$ by a matrix in $\mathcal{X} \subseteq [0,1]^{P \times Q} \subset \mathcal{H} \coloneqq
\mathbb{R}^{P \times Q}$. We denote by $\Omega \in \{0,1\}^{P \times Q}$ a mask such that
$\Omega_{i,j} = 0$ indicates that the pixel at position $(i,j)$ has been damaged, and define the masking operator $R \colon \mathcal{H} \to \mathcal{H},\; Y \mapsto \Omega \odot Y$,
where $\odot$ denotes element-wise multiplication. We note that $R$ is self-adjoint with
operator norm $1$. The corrupt image is $Y_{\mathrm{corrupt}} \coloneqq R(Y)$, and the
recovery problem reads
$$
    \min_{Y \in \mathcal{X}} \left\{
        \frac{1}{2}\norm{R(Y) - Y_{\mathrm{corrupt}}}^2 + \sigma\norm{Y}_*
    \right\},
$$
where $\norm{\cdot}_*$ denotes the nuclear norm, which promotes low-rank structure within the image,
and $\sigma > 0$ is a regularisation parameter. Specifically, we seek the least-norm solution to the above problem, given by
$$
    \min_{Y \in \mathcal{H}} \left\{ \frac{1}{2}\norm{Y}^2 \;:\; Y \in \operatorname*{arg\,min}_{Y \in \mathcal{X}}
        \left\{ \frac{1}{2}\norm{R(Y) - Y_{\mathrm{corrupt}}}^2 + \sigma\norm{Y}_* \right\}
    \right\}.
$$
This matches problem (\ref{eq:MI}) with $\opA = \partial\norm{\cdot}_*$,The proximal step is computed via singular value
soft-thresholding:
$$
    \operatorname{prox}_{\lambda_n \sigma \norm{\cdot}_*}(Z)= U \operatorname{diag}\bigl(\max(s_i - \lambda_n \sigma,\, 0)\bigr) V^\top,
$$
where $Z = U \operatorname{diag}(s) V^\top$ is the singular value decomposition of $Z$. $\opD(Y)= \nabla f(Y)= R^\top(R(Y) -Y_{\mathrm{corrupt}})$, $\opB_{1}= \nabla\Psi$ with $\Psi(Y) =\frac{1}{2}\sum_{i,j}\bigl[\max(Y_{i,j}-1,0)^2 + \min(Y_{i,j},0)^2\bigr]$, which enforces the box constraint $\mathcal{X}$ progressively via the growing weight $\beta_n$.

\paragraph{Parameters Setup.}
To identify the optimal parameter configuration independently of the inertial effect, we first set $\alpha_n=0$. To meet Assumption \ref{A:standing_parameters}, we impose the design criterion $\lambda_n<0.5\mu/\beta_n$. In the numerical experiments, we choose $\lambda_n=0.4\mu/\beta_n$ for non-inertial case, and parameter sequences $\beta_n,\eps_n$ satisfying 
$$
\beta_n=C_{\beta}(1+n)^{e_{\beta}},\qquad \eps_n=\frac{C_{\eps}}{(1+n)^{e_{\eps}}}
$$
since $(\lambda_n)_{n\geq 1}\in\ell^2_{+}(\N)\setminus \ell^1_{+}(\N)$, we need to impose the restriction $1/2<e_{\beta}<1$, meanwhile, $(\lambda_n\eps_n)_{n\geq1}\in \ell^1_{+}(\N)$ applies $e_{\beta}+e_{\eps}>1$.  Additionally $\lambda_n\le\frac{\eps_n}{4(L+\eps_n)^2}$, which implies $e_{\beta}>e_{\eps}$.

Next, we conduct a grid search over the multipliers $C_{\beta}, C_{\varepsilon} \in \{0.1, 1, 10\}$ and the exponents $e_{\beta}, e_{\varepsilon} \in \{0.6, 0.7, 0.8, 0.9, 1.0\}$, subject to the constraint $e_{\beta} > e_{\varepsilon}$, under the setting of 20\% missing pixels and regularisation parameter $\sigma = 50$, consistent with the conditions above. The search reveals that the optimal multipliers are $C_{\beta} = C_{\varepsilon} = 0.1$.
\begin{figure}[H]
\centering
\includegraphics[width=1.0\textwidth]{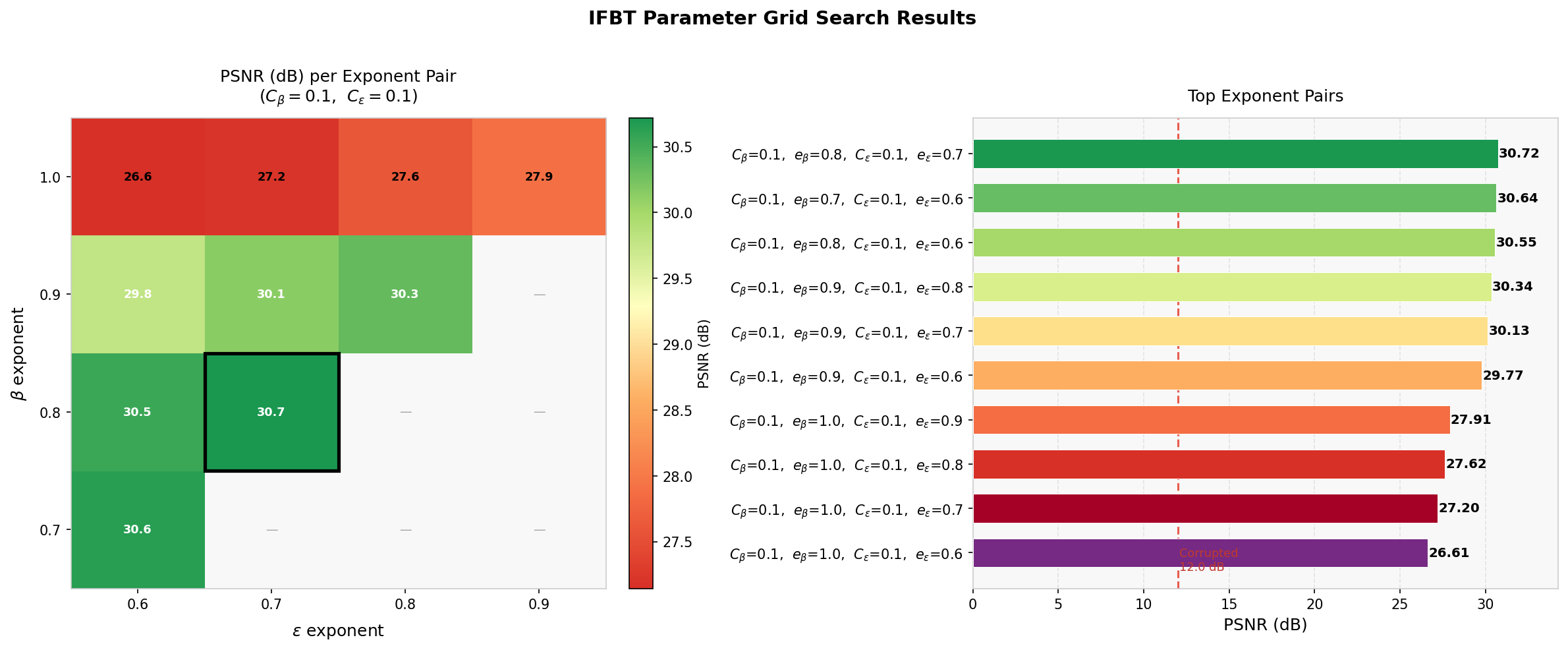}
\caption{Comparison of parameters when $20\%$ corruption }
\label{fig:grid_search}
\end{figure}

 The left side of Figure \ref{fig:grid_search} displays the PSNR heatmap over exponent pairs $(e_{\beta}, e_{\varepsilon})$ at these optimal multipliers, evaluated using the non-inertial algorithm~\textup{(FBT)}, from which the best exponent configuration is identified as $e_{\beta} = 0.8$ and $e_{\varepsilon} = 0.7$. The right side shows the top 10 exponent pairs setting.
 Based on these observations, we fix the parameters for all subsequent inpainting experiments as
$$
    \beta_n = \frac{1}{10}(1+n)^{0.8}, \qquad \varepsilon_n = \frac{1}{10(1+n)^{0.7}}.
$$

\paragraph{Experimental Setup.}
We consider a $256 \times 256$ grayscale image under two experimental settings with different levels of missing data: $20\%$ missing pixels (relatively easy) and $50\%$ missing pixels (more challenging). In both cases, the algorithm is initialised at $X_0 = Y_{\mathrm{corrupt}}$ and we run $N = 2{,}000$ outer iterations.

\subsection{Parameter selection under the theoretical conditions}
In order to meet the requirements formulated in Assumption \ref{A:standing_parameters}, we choose the stepsize as
$$
\lambda_n=\frac{0.49*\mu(1-3\alpha_n)}{\beta_n(1+\alpha_n^2)}
$$
This choice shows that the admissible stepsize decreases as the inertial parameter increases. To avoid an excessively small stepsize while still benefiting from inertial acceleration, we adopt a slowly increasing inertial parameter,
$$\alpha_n=\frac{1}{6}*\frac{n}{N},$$
where $N=2,000$ is the maximum number of iterations, and $\alpha_n$ increases linearly with the iteration number and remains bounded above by $1/6$.
\paragraph{Results.}
We begin by evaluating the proposed algorithm on a $256 \times 256$ grayscale image with $20\%$ of pixels randomly corrupted. The regularisation parameter is set to $\sigma = 50$.
\begin{figure}[H]
\centering
\includegraphics[width=1.0\textwidth]{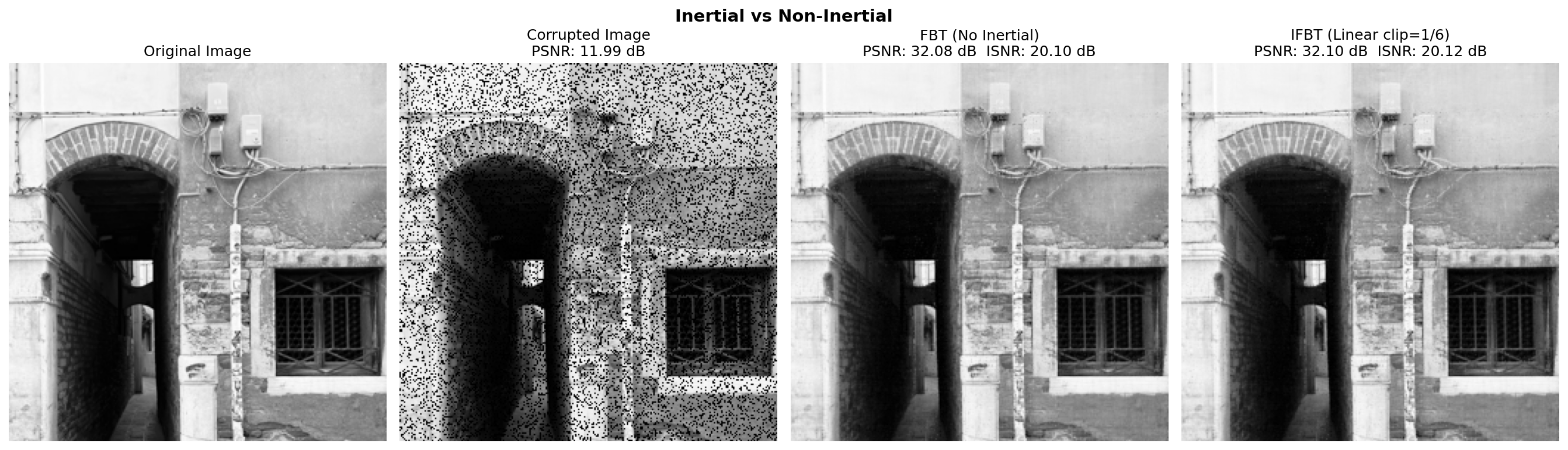}
\caption{Comparison of FBT and IFBT on image inpainting when $20\%$ corruption }
\label{fig:originimages20}
\end{figure}

\begin{figure}[H]
\centering
\includegraphics[width=1.0\textwidth]{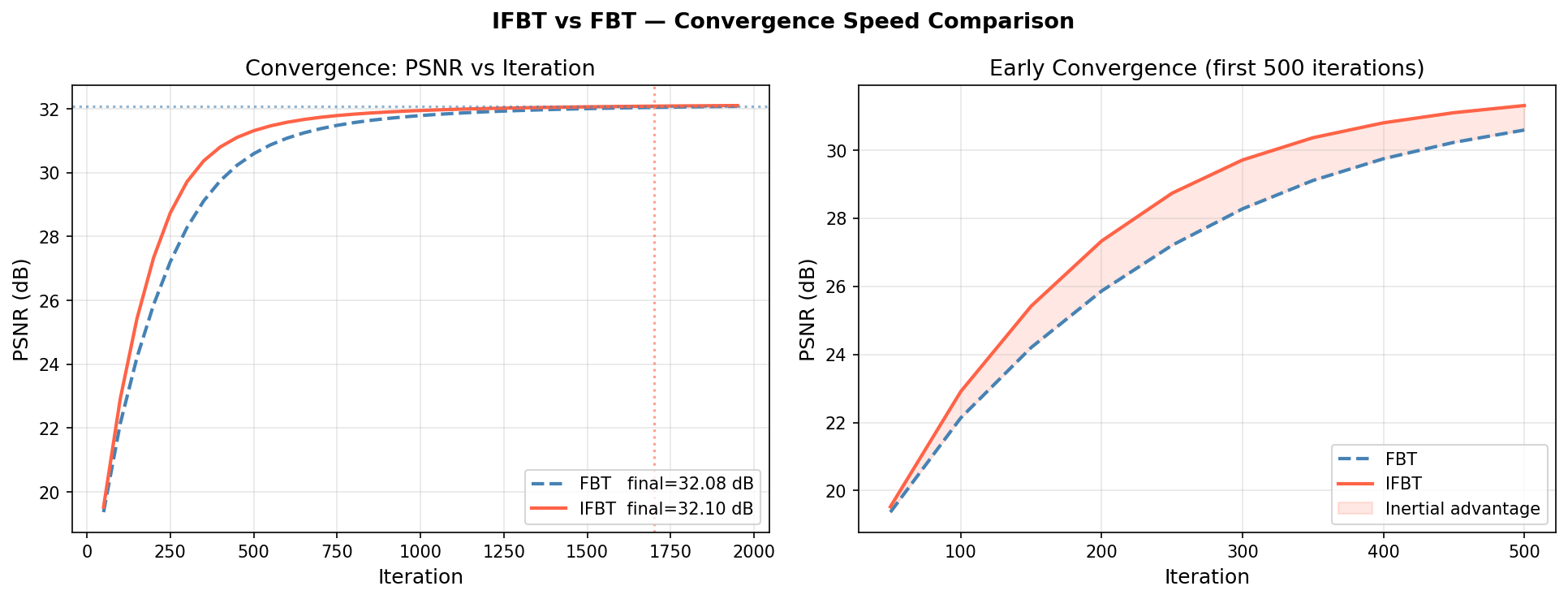}
\caption{Comparison of FBT and IFBT on iterations when $20\%$ corruption }
\label{fig:origincurves20}
\end{figure}
Figure~\ref{fig:originimages20} compares the convergence behavior of IFBT and FBT in terms of PSNR. While both methods achieve high-quality reconstructions, IFBT converges considerably faster. In particular, it reaches the final PSNR attained by FBT $32.08\,\mathrm{dB}$ after $1{,}700$ iterations, corresponding to a reduction of 300 iterations ($15.0\%$), in computational effort.

We further evaluate the algorithm under a more challenging setting, where the corruption rate is increased to $50\%$ of randomly removed pixels and $\sigma = 80$. This higher corruption level serves to stress-test the robustness of the proposed method and to better highlight the performance differences between FBT and IFBT under difficult recovery conditions.
\begin{figure}[H]
\centering
\includegraphics[width=1.0\textwidth]{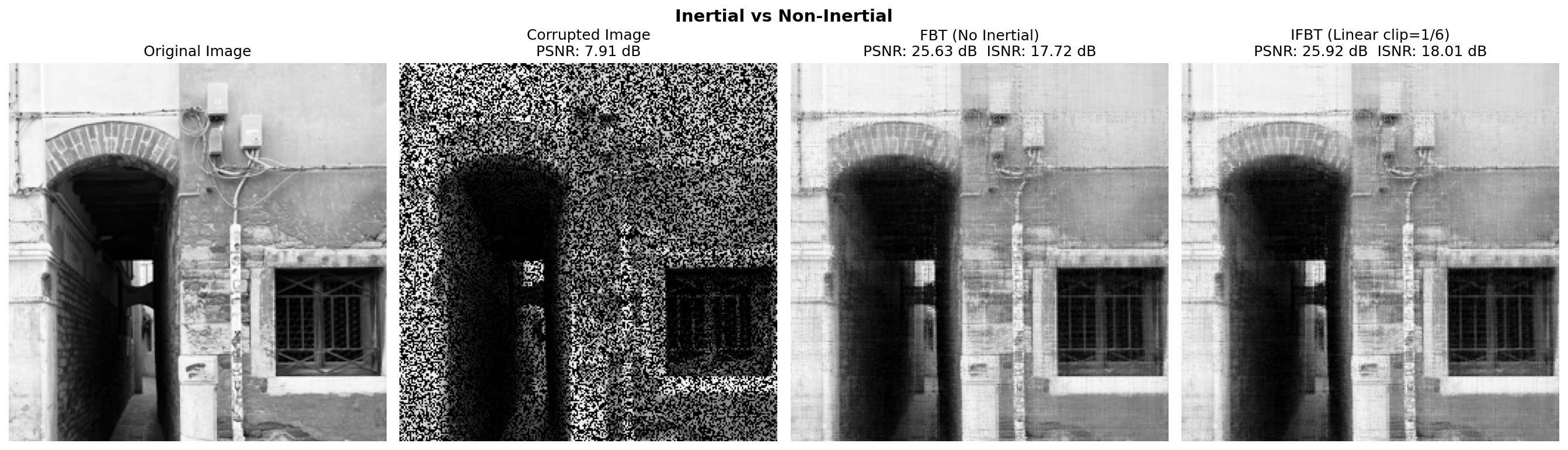}
\caption{Comparison of FBT and IFBT on image inpainting when $20\%$ corruption }
\label{fig:originimages50}
\end{figure}

\begin{figure}[H]
\centering
\includegraphics[width=1.0\textwidth]{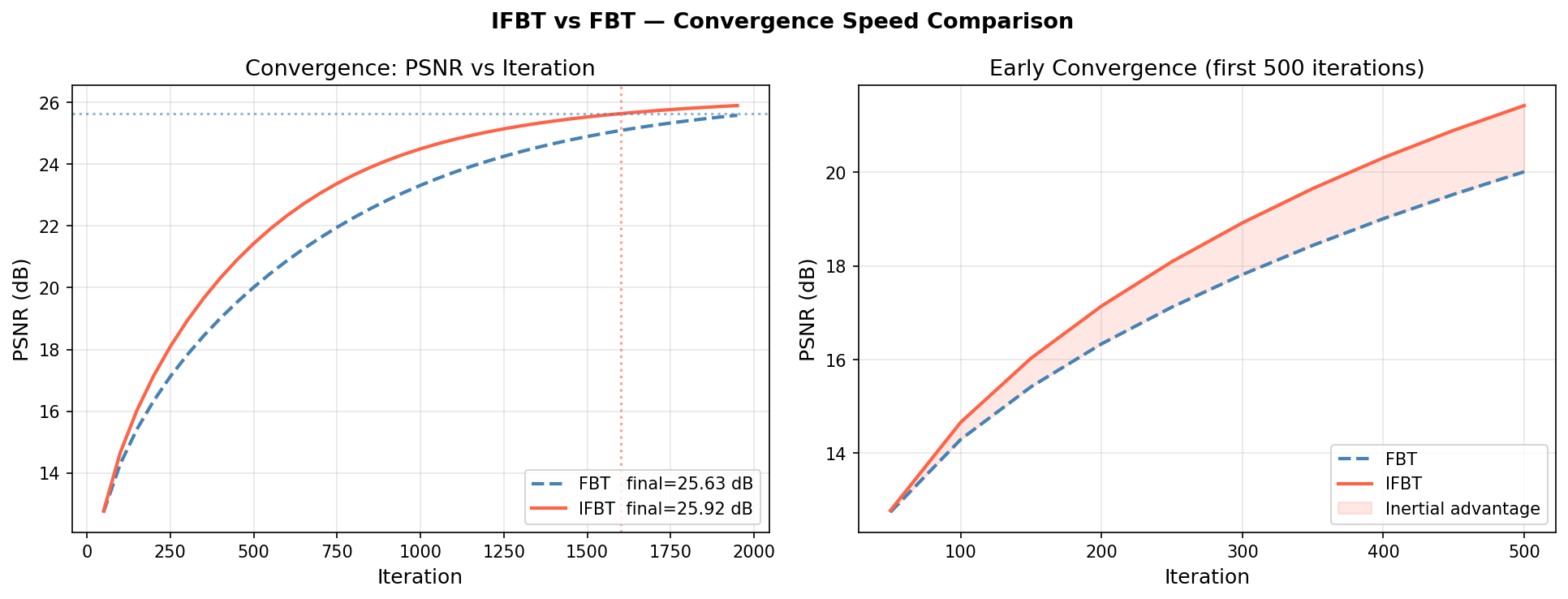}
\caption{Comparison of FBT and IFBT on iterations when $20\%$ corruption }
\label{fig:origincurves50}
\end{figure}
Figure~\ref{fig:origincurves50} reports the reconstruction quality and convergence speed for both methods. Despite the significantly higher corruption level, IFBT converges considerably faster. In particular, it reaches the final PSNR attained by FBT $25.63\,\mathrm{dB}$ after only $1{,}600$ iterations, corresponding to a reduction of 400 iterations ($20.0\%$), in computational effort.

\subsection{Decoupled inertial acceleration}
Although the theoretical analysis requires the stepsize and inertial parameters to be coupled, we notice better numerical performance is obtained by decoupling them. Accordingly, we apply inertia only in the extrapolation step $y_n$ and set the inertial parameter to 0 in the stepsize. The resulting performance is reported below.

In this case, since we decoupled them, the growth of $\alpha_n$ will not affect $\lambda_n$ anymore, we adopt a Nesterov-type inertial coefficient defined as
$$
\alpha_n = \min\left(\frac{9}{30}, \max\left(0, \frac{t_n - 1}{t_{n+1}}\right)\right),
$$
where the auxiliary sequence $\{t_n\}$ is given by
$$
t_{n} = \frac{1 + \sqrt{1 + 4 (n+1)^2}}{2}, \quad n \ge 0.
$$
This ensures $\alpha_n$ is non-decreasing and confined to $(0, 1/3)$. 

\paragraph{Results.}
We begin by evaluating the proposed algorithm on a $256 \times 256$ grayscale image with $20\%$ of pixels randomly corrupted. The regularisation parameter is set to $\sigma = 50$.

\begin{figure}[H]
\centering
\includegraphics[width=1.0\textwidth]{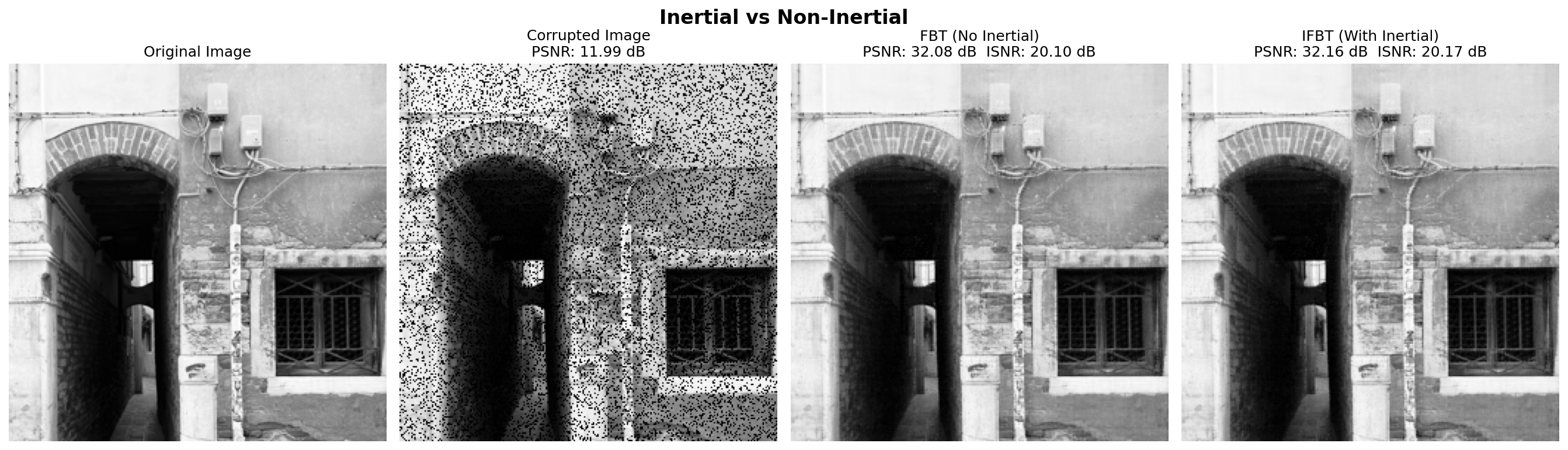}
\caption{Comparison of FBT and IFBT on image inpainting when $20\%$ corruption }
\label{fig:images20}
\end{figure}

\begin{figure}[H]
\centering
\includegraphics[width=1.0\textwidth]{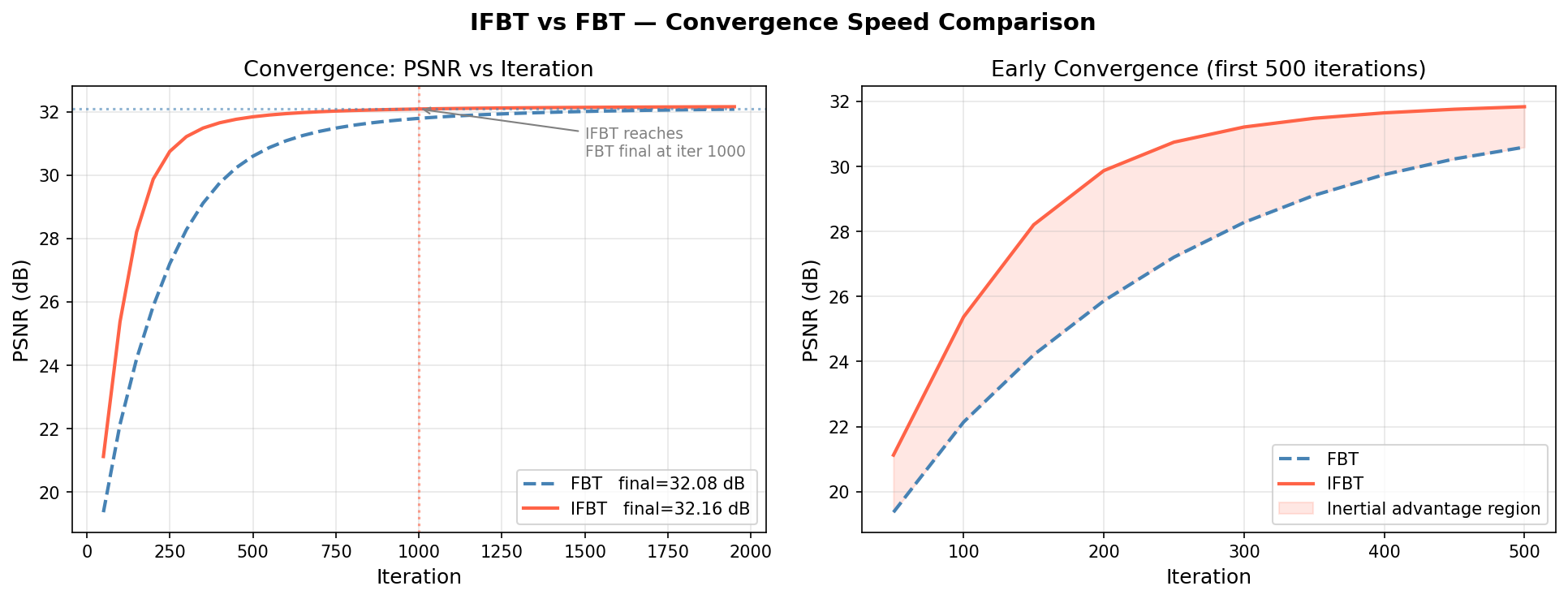}
\caption{Comparison of FBT and IFBT on iterations when $20\%$ corruption }
\label{fig:curves20}
\end{figure}

Figure~\ref{fig:curves20} reports the reconstruction quality in terms of PSNR, together with the convergence speedup of IFBT over FBT. Both methods recover the image above $30\,\mathrm{dB}$, confirming the effectiveness of nuclear norm regularisation for inpainting. IFBT achieves a higher final PSNR while reaching the FBT final quality level $1{,}000$ iterations earlier, corresponding to a $50.0\%$ saving in computational effort.

We further evaluate the algorithm under a more challenging setting, where the corruption rate is increased to $50\%$ of randomly removed pixels and $\sigma = 80$. 

\begin{figure}[H]
\centering
\includegraphics[width=1.0\textwidth]{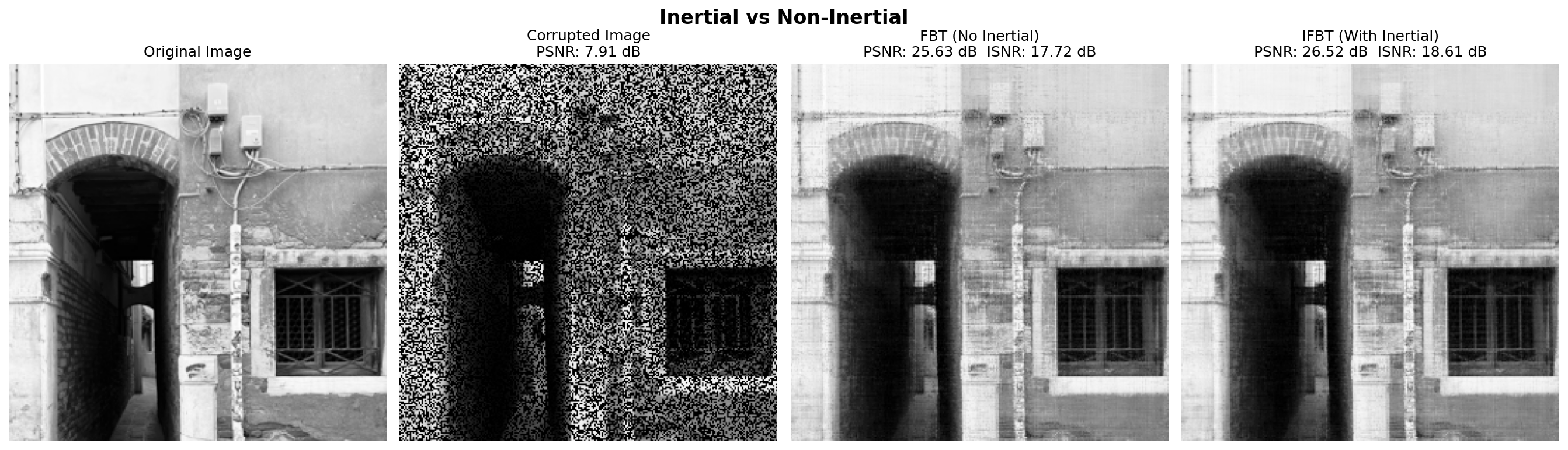}
\caption{Comparison of FBT and IFBT on image inpainting when $50\%$ corruption }
\label{fig:images50}
\end{figure}

\begin{figure}[H]
\centering
\includegraphics[width=1.0\textwidth]{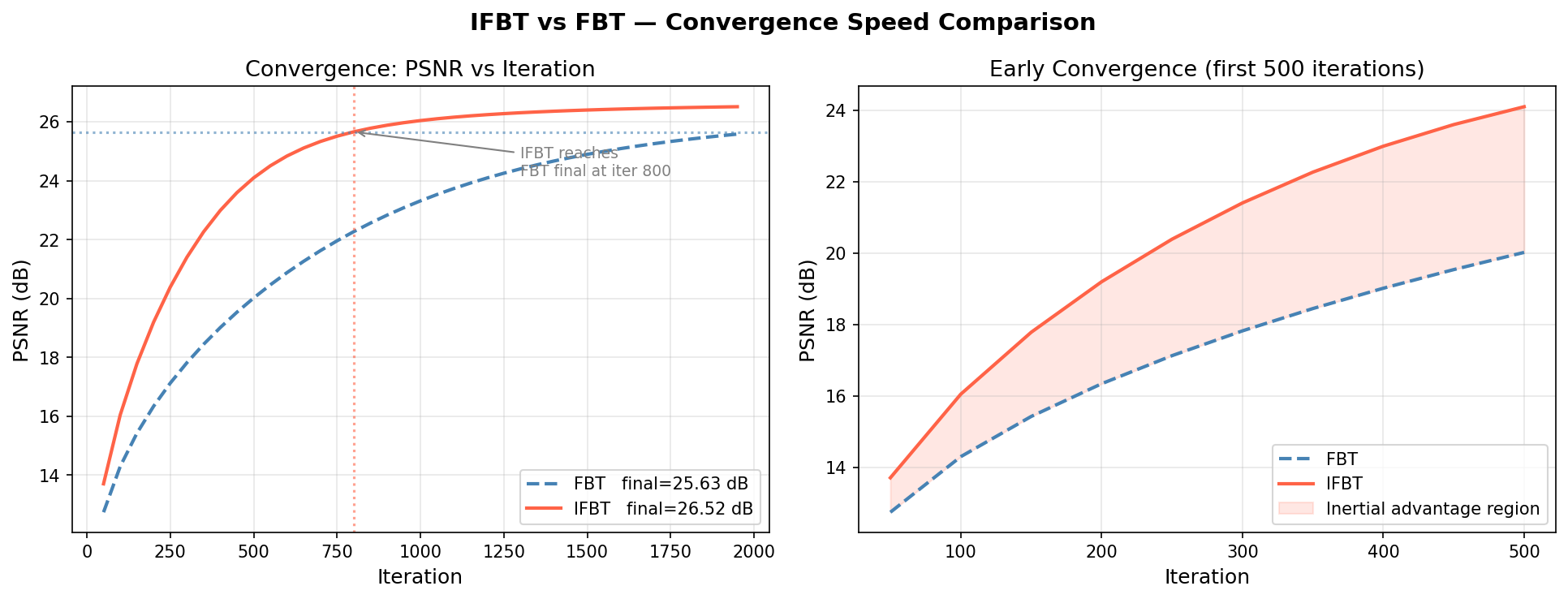}
\caption{Comparison of FBT and IFBT on iterations when $50\%$ corruption }
\label{fig:curves50}
\end{figure}

Figure \ref{fig:curves50} reports the reconstruction quality and convergence speed for both methods. Despite the significantly higher corruption level, IFBT achieves a higher final PSNR. Notably, IFBT reaches the quality level obtained by FBT after 2,000 iterations at 800 iterations, corresponding to a $60.0\%$ reduction in computational effort. This further demonstrates the advantage of inertial acceleration under challenging recovery conditions.

From the results, the inertial term significantly accelerates the convergence process. In particular, IFBT achieves
better reconstruction quality than FBT with approximately half the number of iterations. Moreover, the improvement is particularly pronounced in the early iterations, where the inertial effect rapidly guides the solution toward a high-quality estimate. 

These observations suggest that the inertial mechanism is particularly beneficial when the reconstruction problem becomes more difficult. When fewer pixels are observed, the baseline method converges more slowly, while the inertial term helps accelerate the iterative updates. These results demonstrate the effectiveness of the inertial mechanism in improving both convergence speed and reconstruction performance.

\section{Conclusion}
\label{sec:conclusion}
In this paper we have studied forward-backward penalty methods for solving a general system of constrained variational inequalities. Introducing inertial effects to potentially boost the numerical performance of the method, together with Tikhonov regularization to promote strong convergence, we give a general analysis on the 
long-run behavior of the algorithmic sequence. Invoking the celebrated Attouch-Czarnecki condition to deal with the penalty terms, we prove a general weak convergence theorem on the ergodic sequence. Adding an additional non-degenericity assumption on the data of the problem, we prove strong convergence of the last iterate of the numerical scheme.
Future research should study in depth the quantitative implications of inertia. Second, it seems to be possible to extend some of the results to the stochastic regime, in the spirit of \cite{luke2026asymptotic}. We leave these investigations for future research.

\paragraph{Acknowledgements.} {This research benefited from the support of the FMJH Program PGMO. MST's research is supported by the Deutsche Forschungsgemeinschaft (DFG) - Projektnummer 556222748 "non-stationary hierarchical minimization". Siqi Qu is supported by a CSC scholarship. }

%

\bibliographystyle{plain}
\bibliography{PenaltyDynamics}

@article{attouch2011prox,
  title={Prox-penalization and splitting methods for constrained variational problems},
  author={Attouch, H{\'e}dy and Czarnecki, Marc-Olivier and Peypouquet, Juan},
  journal={SIAM Journal on Optimization},
  volume={21},
  number={1},
  pages={149--173},
  year={2011},
  publisher={SIAM}
}

@article{hintermuller2015generalized,
  title={Generalized Nash equilibrium problems in Banach spaces: Theory, Nikaido--Isoda-based path-following methods, and applications},
  author={Hintermüller, M and Surowiec, Thomas and Kämmler, A},
  journal={SIAM Journal on Optimization},
  volume={25},
  number={3},
  pages={1826--1856},
  year={2015},
  publisher={SIAM}
}

@article{garrido2026stochastic,
  title={Stochastic Differential Inclusions driven by Maximal Monotone Operators with empty interiors},
  author={Garrido, Juan Guillermo and P{\'e}rez-Aros, Pedro and Staudigl, Mathias},
  journal={arXiv preprint arXiv:2602.23145},
  year={2026}
}

@article{rascanu2014deterministic,
  title={Deterministic and stochastic differential equations in Hilbert spaces involving multivalued maximal monotone operators},
  author={Rascanu, Aurel},
  journal={arXiv preprint arXiv:1402.0748},
  year={2014}
}

@article{luke2026asymptotic,
  title={Asymptotic behaviour of coupled random dynamical systems with multiscale aspects},
  author={Luke, D Russell and Schnebel, Johannes-Carl and Staudigl, Mathias and Peypouquet, Juan and Qu, Siqi},
  journal={arXiv preprint arXiv:2601.15411},
  year={2026}
}

@article{dreves2018generalized,
  title={A generalized Nash equilibrium approach for optimal control problems of autonomous cars},
  author={Dreves, Axel and Gerdts, Matthias},
  journal={Optimal Control Applications and Methods},
  volume={39},
  number={1},
  pages={326--342},
  year={2018},
  publisher={Wiley Online Library}
}

@article{kanzow2019multiplier,
  title={The multiplier-penalty method for generalized Nash equilibrium problems in Banach spaces},
  author={Kanzow, Christian and Karl, Veronika and Steck, Daniel and Wachsmuth, Daniel},
  journal={SIAM Journal on Optimization},
  volume={29},
  number={1},
  pages={767--793},
  year={2019},
  publisher={SIAM}
}

@article{Banert:2015aa,
	author = {Banert, Sebastian and Bo{\c t}, Radu Ioan},
	date = {2015/09/01},
	doi = {10.1007/s10957-014-0700-x},
	id = {Banert2015},
	isbn = {1573-2878},
	journal = {Journal of Optimization Theory and Applications},
	number = {3},
	pages = {930--948},
	title = {Backward Penalty Schemes for Monotone Inclusion Problems},
	url = {https://doi.org/10.1007/s10957-014-0700-x},
	volume = {166},
	year = {2015}}

@article{peypouquet2026asymptotic,
	author = {Peypouquet, Juan and Qu, Siqi and Staudigl, Mathias},
	journal = {Applied Mathematics \& Optimization},
	number = {2},
	pages = {36},
	publisher = {Springer},
	title = {Asymptotic Behavior of Penalty Dynamics for Constrained Variational Inequalities},
	volume = {93},
	year = {2026}}

@article{AttCzar18,
	author = {Attouch, Hedy and Cabot, Alexandre and Czarnecki, Marc-Olivier},
	journal = {Transactions of the American Mathematical Society},
	number = {2},
	pages = {755--790},
	title = {Asymptotic behavior of nonautonomous monotone and subgradient evolution equations},
	volume = {370},
	year = {2018}}

@article{Bot:2020aa,
	author = {Bo{\c t}, Radu Ioan and Grad, Sorin-Mihai and Meier, Dennis and Staudigl, Mathias},
	isbn = {2191-950X},
	journal = {Advances in Nonlinear Analysis},
	number = {1},
	pages = {450--476},
	publisher = {de Gruyter},
	title = {Inducing strong convergence of trajectories in dynamical systems associated to monotone inclusions with composite structure},
	volume = {10},
	year = {2020}}

@article{de2011optimal,
	author = {De los Reyes, Juan Carlos},
	journal = {SIAM Journal on Control and Optimization},
	number = {4},
	pages = {1629--1658},
	publisher = {SIAM},
	title = {Optimal control of a class of variational inequalities of the second kind},
	volume = {49},
	year = {2011}}

@article{noun2013forward,
	author = {Noun, Nahla and Peypouquet, Juan},
	journal = {Journal of Optimization Theory and Applications},
	pages = {787--795},
	publisher = {Springer},
	title = {Forward--backward penalty scheme for constrained convex minimization without inf-compactness},
	volume = {158},
	year = {2013}}

@article{boct2018second,
	author = {Bo{\c{t}}, Radu Ioan and Csetnek, Ern{\"o} Robert and L{\'a}szl{\'o}, Szil{\'a}rd Csaba},
	journal = {Analysis and Applications},
	number = {05},
	pages = {601--622},
	publisher = {World Scientific},
	title = {Second-order dynamical systems with penalty terms associated to monotone inclusions},
	volume = {16},
	year = {2018}}

@article{cabot2005,
	author = {Cabot, Alexandre},
	journal = {SIAM Journal on Optimization},
	number = {2},
	pages = {555--572},
	publisher = {SIAM},
	title = {Proximal point algorithm controlled by a slowly vanishing term: applications to hierarchical minimization},
	volume = {15},
	year = {2005}}

@book{nesterov2018lectures,
  title={Lectures on convex optimization},
  author={Nesterov, Yurii and others},
  volume={137},
  year={2018},
  publisher={Springer}
}

@article{de2016strong,
	author = {De los Reyes, Juan Carlos and Meyer, Christian},
	journal = {Journal of Optimization Theory and Applications},
	pages = {375--409},
	publisher = {Springer},
	title = {Strong stationarity conditions for a class of optimization problems governed by variational inequalities of the second kind},
	volume = {168},
	year = {2016}}

@article{sabach2017first,
	author = {Sabach, Shoham and Shtern, Shimrit},
	journal = {SIAM Journal on Optimization},
	number = {2},
	pages = {640--660},
	publisher = {SIAM},
	title = {A first order method for solving convex bilevel optimization problems},
	volume = {27},
	year = {2017}}

@article{Bot:2019aa,
	author = {Bo{\c t},Radu Ioan and Nguyen ,Dang-Khoa},
	date = {2019/10/03},
	doi = {10.1080/02331934.2018.1556662},
	isbn = {0233-1934},
	journal = {Optimization},
	journal1 = {Optimization},
	journal2 = {Optimization},
	month = {10},
	number = {10},
	pages = {1855--1880},
	publisher = {Taylor \& Francis},
	title = {A forward--backward penalty scheme with inertial effects for monotone inclusions. Applications to convex bilevel programming},
	type = {doi: 10.1080/02331934.2018.1556662},
	url = {https://doi.org/10.1080/02331934.2018.1556662},
	volume = {68},
	year = {2019},
	year1 = {2019}}

@article{Bot:2016aa,
	author = {Bo{\c t}, Radu Ioan and Csetnek, Ern{\"o}Robert},
	date = {2016/03/15/},
	doi = {https://doi.org/10.1016/j.jmaa.2015.11.032},
	isbn = {0022-247X},
	journal = {Journal of Mathematical Analysis and Applications},
	number = {2},
	pages = {1688--1700},
	title = {Approaching the solving of constrained variational inequalities via penalty term-based dynamical systems},
	url = {https://www.sciencedirect.com/science/article/pii/S0022247X15010707},
	volume = {435},
	year = {2016}}

@article{Solodov:2007aa,
	author = {Solodov, Mikhail},
	isbn = {0944-6532},
	journal = {Journal of Convex Analysis},
	number = {2},
	pages = {227},
	publisher = {HELDERMANN VERLAG LANGER GRABEN 17, 32657 LEMGO, GERMANY},
	title = {An explicit descent method for bilevel convex optimization},
	volume = {14},
	year = {2007}}

@article{Attouch:2010aa,
	author = {Attouch, Hedy and Czarnecki, Marc-Olivier},
	isbn = {0022-0396},
	journal = {Journal of Differential Equations},
	number = {6},
	pages = {1315--1344},
	publisher = {Elsevier},
	title = {Asymptotic behavior of coupled dynamical systems with multiscale aspects},
	volume = {248},
	year = {2010}}

@article{attouch1996dynamical,
	author = {Attouch, Hedy and Cominetti, Roberto},
	journal = {Journal of Differential Equations},
	number = {2},
	pages = {519--540},
	publisher = {Elsevier},
	title = {A dynamical approach to convex minimization coupling approximation with the steepest descent method},
	volume = {128},
	year = {1996}}

@article{Peypouquet:2012aa,
	author = {Peypouquet, Juan},
	date = {2012/04/01},
	doi = {10.1007/s10957-011-9936-x},
	id = {Peypouquet2012},
	isbn = {1573-2878},
	journal = {Journal of Optimization Theory and Applications},
	number = {1},
	pages = {123--138},
	title = {Coupling the Gradient Method with a General Exterior Penalization Scheme for Convex Minimization},
	url = {https://doi.org/10.1007/s10957-011-9936-x},
	volume = {153},
	year = {2012}}

@book{BauCom16,
	address = {CMS Books in Mathematics},
	author = {Heinz H. Bauschke and Patrick L. Combettes},
	publisher = {Springer},
	title = {Convex Analysis and Monotone Operator Theory in Hilbert Spaces},
	year = {2016}}

@article{Bot:2014aa,
	author = {Bo{\c t}, Radu Ioan and Csetnek, Ern{\"o}Robert},
	date = {2014/06/01},
	doi = {10.1007/s11228-014-0274-7},
	id = {Bo{\c t}2014},
	isbn = {1877-0541},
	journal = {Set-Valued and Variational Analysis},
	number = {2},
	pages = {313--331},
	title = {Forward-Backward and Tseng's Type Penalty Schemes for Monotone Inclusion Problems},
	url = {https://doi.org/10.1007/s11228-014-0274-7},
	volume = {22},
	year = {2014}}

@article{AttCzarPey11,
	author = {Attouch, H\'{e}dy and Czarnecki, Marc-Olivier and Peypouquet, Juan},
	doi = {10.1137/110820300},
	eprint = {https://doi.org/10.1137/110820300},
	journal = {SIAM Journal on Optimization},
	number = {4},
	pages = {1251-1274},
	title = {Coupling Forward-Backward with Penalty Schemes and Parallel Splitting for Constrained Variational Inequalities},
	url = {https://doi.org/10.1137/110820300},
	volume = {21},
	year = {2011}}

@article{cortild2025regularization,
  author = {Daniel Cortild and Meggie Marschner and Mathias Staudigl},
  title = {Regularization Methods for Solving Hierarchical Variational Inequalities with Complexity Guarantees},
  journal = {arXiv preprint arXiv:2512.20772},
  year = {2025}
}

\end{document}